\documentclass{amsart}
\usepackage{graphicx} 

\usepackage{bbm,amsmath, amsfonts, amsthm, amssymb}
\usepackage{bm}
\usepackage[margin=3cm]{geometry}
\usepackage{xcolor}
\usepackage{graphicx,tabularx}
\usepackage{enumitem}

\usepackage{mathtools}
\graphicspath{{Figures/}}

\usepackage{setspace}
\usepackage[export]{adjustbox}

\usepackage[textsize=tiny]{todonotes}
\newcommand{\EE}{\mathbb{E}}

\newcommand{\NN}{\mathbb{N}}
\newcommand{\PP}{\mathbb{P}}
\newcommand{\RR}{\mathbb{R}}
\newcommand{\D}{\mathrm{d}}
\newcommand{\ds}{\mathrm{d}s}
\newcommand{\dr}{\mathrm{d}r}
\newcommand{\dt}{\mathrm{d}t}
\newcommand{\du}{\mathrm{d}u}

\newcommand{\E}{\mathrm{e}}
\newcommand{\Cc}{\mathcal{C}}
\newcommand{\Dd}{\mathcal{D}}
\newcommand{\Ee}{\mathcal{E}}
\newcommand{\Ff}{\mathcal{F}}

\newcommand{\Oo}{\mathcal{O}}

\newcommand{\Ww}{\mathcal{W}}

\newcommand{\ep}{\varepsilon}

\newcommand{\frf}{\mathfrak{f}}
\newcommand{\frF}{\mathfrak{F}}

\newcommand{\VIX}{{\rm VIX}}

\newcommand{\Wh}{\widehat{W}}
\newcommand{\Th}{\widehat{\Theta}}

\newcommand{\kphi}{\kappa_{\phi}}
\newcommand{\bxb}{\bar{\bx}}
\newcommand{\etab}{\bar{\eta}}
\newcommand{\tb}{\bar{t}}
\newcommand{\Tb}{\breve{T}}
\newcommand{\wS}{\widehat{\Sigma}}
\newcommand{\bt}{\mathbf{t}}
\newcommand{\bw}{\mathbf{w}}
\newcommand{\bx}{{\bm{x}}}

\newcommand{\half}{\frac{1}{2}}
\newcommand{\one}{\mathbbm{1}}

\newcommand{\abs}[1]{\left\lvert#1\right\rvert}
\newcommand{\norm}[1]{\left\lVert#1\right\rVert}
\newcommand{\norminfty}[1]{\left\lVert#1\right\rVert_{\infty}}

\newcommand{\BS}{{\rm BS}}

\newcommand{\notthis}[1]{}
\newtheorem{theorem}{Theorem}[section]
\newtheorem{corollary}[theorem]{Corollary}
\newtheorem{lemma}[theorem]{Lemma}
\newtheorem{proposition}[theorem]{Proposition}

\theoremstyle{definition}
\newtheorem{definition}[theorem]{Definition}
\newtheorem{remark}[theorem]{Remark}
\newtheorem{assumption}[theorem]{Assumption}
\newtheorem{example}[theorem]{Example}
\numberwithin{equation}{section}

\definecolor{ocean}{rgb}{0,0.1,0.6}

\definecolor{imperialGreen}{RGB}{2,137,59}
\definecolor{imperialBlue}{RGB}{0, 62, 116}
\definecolor{imperialBrick}{RGB}{165,25,0}
\definecolor{imperialProcess}{RGB}{0,133,202}

\usepackage{hyperref}
\hypersetup{colorlinks=false, linkbordercolor=imperialProcess,
citebordercolor=imperialBrick}

\usepackage[foot]{amsaddr} 

\author{Alexandre Pannier}
\address{Université Paris Cité, Laboratoire de Probabilités Statistique et Modélisation (LPSM)}
\email{pannier@lpsm.paris}

\title[\sc{Path-dependent PDEs for volatility derivatives}]{\Large{P\lowercase{ath-dependent} PDE\lowercase{s for volatility derivatives}}}

\date{\today}

\thanks{The author would like to thank Rigpa Thapa for her unwavering support and source of inspiration.}

\begin{document}

\begin{abstract}
    We regard options on VIX and Realised Variance as solutions to path-dependent partial differential equations (PDEs) in a continuous stochastic volatility model. 
    The modeling assumption specifies that the instantaneous variance is a $C^3$ function of a multidimensional Gaussian Volterra process; this includes a large class of models suggested for the purpose of VIX option pricing, either rough, or not, or mixed.
    We unveil the path-dependence of those volatility derivatives and, under a regularity hypothesis on the payoff function, we prove the well-posedness of the associated PDE. The latter is of heat type, because of the Gaussian assumption, and the terminal condition is also path-dependent. Furthermore, formulae for the greeks are provided, the implied volatility is shown to satisfy a quasi-linear path-dependent PDE and, in Markovian models, finite-dimensional pricing PDEs are obtained for VIX options.
\end{abstract}

\maketitle

\section{Introduction}

In a continuous time model, VIX and Realised Variance (RV) both boil down to time-averages of the stochastic volatility of the asset. Prices of derivatives on these underlyings are represented as expectations, hence their numerical evaluation naturally leans towards Monte Carlo methods. The extensive literature dedicated to simulation schemes---which covers their design, numerical implementation, and convergence analysis---bears witness to the omnipresence of this approach, especially for volatility derivatives. 
As an alternative, this paper proposes to view option prices on volatility derivatives such as the VIX as solutions to a path-dependent PDE (PPDE).
Before getting into more details, we should first explore the motivation behind this class of financial assets.
Readers who wish to get down to business may directly jump to Section~\ref{subsec:PPDEs}.

\subsection{Background}
Volatility derivatives are used both as risk management and speculation tools to get a direct exposure to an index or a stock's volatility. Although this class has attracted attention as a whole, VIX derivatives became some of the most liquid instruments on the financial markets. The CBOE Volatility Index (VIX) measures the 30-day forward-looking volatility of the S\&P500 index (SPX); more precisely, it represents a log-contract on the SPX, approximately replicated with a weighted sum of quoted Calls and Puts. In an idealised stochastic volatility model $\D S_t/S_t = \sigma_t \D B_t$, we have by Itô's formula
\begin{equation}\label{eq:defVIX}
\VIX_T^2 = \EE\left[-\frac{2}{\Delta} \log\left(\frac{S_{T+\Delta}}{S_T}\right) \Big\lvert \Ff_T \right]
= \frac{1}{\Delta}\int_T^{T+\Delta} \EE\big[\sigma^2_t | \Ff_T \big] \dt,
\end{equation}
where $\Delta$ is the 30-day window and~$\EE[\sigma^2_t|\Ff_T]$ represents the forward variance curve. 
The structural links between VIX and SPX, visible from the computation above, require a consistent model for $\sigma$ able to jointly calibrate options on both underlyings simultaneously. This central issue is a driving force of research on volatility and has proven particularly challenging to resolve. Since the introduction of VIX options in 2006, massive efforts were produced by the community to design models and numerical methods up to the task. 

Bergomi \cite{bergomi2008smile3} and Gatheral \cite{gatheral2008consistent} rapidly argued for multifactor models (Ornstein-Uhlenbeck and CEV respectively), before jump models tackled the problem. We will gently leap over this stream of research as we are concerned here with continuous and particularly rough volatility models. Let us mention that robust statistical estimators and significant empirical evidence~\cite{bolko2022gmm,chong2024statisticala,chong2024statisticalb,fukasawa2022consistent,han2023estimating,wu2022rough} confirmed the original thesis that log-volatility trajectories indeed have low regularity \cite{gatheral2018volatility}. Evidence from options data can also be found in~\cite{BFG16,delemotte2023yet,guyon2022does,romer2022empirical}. More relevant to this paper's premise, a variety of simulation schemes were developped around and away from the traditional Euler scheme, including the hybrid scheme~\cite{bennedsen2017hybrid,mccrickerd2018turbocharging}, a tree formulation \cite{horvath2024functional}, Markovian multifactor approximations \cite{abi2019multifactor,alfonsi2023approximation,bayer2023markovian} and quantization~\cite{bonesini2023functional}. When it comes to pricing VIX options, these models also went through the Monte Carlo pipe~\cite{bourgey2021multilevel,gatheral2020quadratic,horvath2020volatility,jacquier2018vix,rosenbaum2021deep}. Indeed, rough volatility has not been generous on alternatives: a few asymptotic results~\cite{alos2018smile,forde2021small,jacquier2021rough,lacombe2021asymptotics} and a weak expansion~\cite{bourgey2023weak} are the only contenders to the best of our knowledge.

Widening our scope a little, we observe that signature-based models led to a semi-closed formula for VIX~\cite{cuchiero2025joint}. When the variance is a semimartingale with linear drift, one can write $\VIX_T$ as a function of $\sigma_T$ only, which leads to a standard type of pricing PDE. The authors of~\cite{barletta2019short,fouque2018heston,lin2009vix,papanicolaou2014regime} exploited this property to derive fast pricing techniques and small-parameter expansions in Heston type of models. Similarly, under the assumption that the variance curve is of the form~$G(Z_t,T-t)$ with~$Z$ a semimartingale, Buehler~\cite{buehler2006consistent} proved an HJM condition for the variance curve which takes the form of a PDE. This is again the case in the Heston model. Unfortunately, the non-Markovianity of the (quadratic) rough Heston model prevents both these ideas to be applied there.
Recently, several multifactor models, dubbed quintic~\cite{abi2025joint}, S-M-2F-QHyp~\cite{romer2022empirical} and 4-factor path-dependent volatility~\cite{guyon2023volatility}, claimed that simple transformations of Markovian dynamics were sufficient to capture the joint calibration. The first two, in particular, model the volatility as transformations of Ornstein-Uhlenbeck (OU) processes, similarly to what Bergomi originally proposed in~\cite{bergomi2008smile2,bergomi2008smile3}. This further motivates us to model the squared volatility as a generic function of a multidimensional Gaussian Volterra process $\Wh$, such as a fractional Brownian motion or an OU process. Our framework thus encompasses many of the models presented above, in particular the family of Bergomi models which can be of multifactor, rough or mixed type or any combination of the aforementioned (see Examples~\ref{ex:kernels} and~\ref{ex:volfct}). Suitable and more stringent conditions should allow to generalise our results to a broader class of processes including solutions to stochastic Volterra equations. However, none of the above-mentioned models would benefit from this extension (not even the quadratic rough Heston model) hence we refrain from doing so.  

\subsection{Path-dependent PDEs}\label{subsec:PPDEs}
Options with path-dependent payoffs are expected since the seminal work of Dupire~\cite{dupire2019functional} to satisfy a certain kind of PPDE. However, to the best of our knowledge, they were never formulated in the context of variance options. Cont and Fournié~\cite{cont2010change,cont2013functional} provided a rigorous framework for the functional Itô formula of Dupire and extended it to Dirichlet processes, which include in particular the integral~\eqref{eq:defVIX} as a function of $T$. With a Markovian variance process, variance options may still be represented as functions of a Dirichlet process. This representation unfortunately fails in general in the non-Markovian case this paper focuses on: it is the central observation of Viens and Zhang~\cite[Section 2]{viens2019martingale}. Moreover, the PPDEs resulting from the Dupire-Cont-Fournié calculus were studied primarly via the lens of viscosity solutions~\cite[Chapter 8]{bally2016stochastic}, \cite{ren2014overview}. The only instance of a classical solution arose in the specific case where the underlying is a Brownian motion~\cite{peng2016bsde}.

The case of rough (non-Markovian) volatility, more recent, was initiated by the functional Itô formula derived in~\cite{viens2019martingale}, and the PDE aspect was further developped in~\cite{bonesini2023rough,wang2022path}.
The insight of the former authors is that, contrary to the Markovian case, if $(\Wh_t:=\int_0^t K(t,s)\D W_s)_{t\ge0}$ is an~$\RR^d$-valued Gaussian Volterra process then the conditional expectation~$\displaystyle\EE\big[\varphi\big(\Wh_{[0,T]}\big) \big| \Ff_t\big]$ not only depends on the past trajectory~$\Wh_{[0,t]}$ but also on the ``forward curve"~$(\Theta^t_s:=\EE[\Wh_s |\Ff_t])_{s\in[t,T]}$. This infinite-dimensional path encodes all the necessary information to recover the Markov property in the space of continuous paths~$C^0([0,T],\RR^d)$. The functional Itô formula is then established for processes of the type~$u(t,\Wh \otimes_t\Theta^t)$ where~$(\bx\otimes_t\theta)_s := \bx_s \one_{t>s} + \theta_s \one_{t\le s}$. It involves Fréchet derivatives in $C^0([t,T])$ and thereby only perturbs the path~$\Theta^t$. This choice of state space seems natural but has a notable drawback: for Volterra processes (and rough volatility models) the direction of the derivative turns out to be~$K(\cdot,t)$, which is not continuous over~$[t,T]$. This singularity has to be circumvented via an approximation argument, as described in~\cite{viens2019martingale}, which imposes stringent conditions on~$u$ and its derivatives.
Thus, the operator appearing in the functional Itô formula is the limit of a sequence of Fréchet derivatives.  
We note for completeness that the functional Itô formula was actually proved for a much larger class of Volterra processes. 

The framework of this paper, while natural, goes beyong the assumptions made in the aforementioned works. The non-Markovian nature of the underlying prevents the use of the functional calculus of~\cite{dupire2019functional,cont2010change}. Well-posedness results for the  Volterra-type PPDE were obtained in~\cite{viens2019martingale} and~\cite{bonesini2023rough} for state-dependent payoffs and under specific models, a Gaussian Volterra process and the log-price in the rough Bergomi model, respectively.  
The setup of variance options, which involves a path-dependent payoff and a conditional expectation, calls for a new type of representation of the option price stated in Proposition~\ref{prop:rpz}. As a result, the derivation of the PPDE differs from prior studies.

\subsection{Contributions}
In this work, we are interested in option prices of the form~$\EE[\phi(V_T)|\Ff_t]$, where~$V_T$ is a volatility derivative (think~${\rm VIX}^2_T$ as in~\eqref{eq:defVIX}) and $\sigma_t^2=f_t(\Wh_t)$.
Our contributions are threefold: 
\begin{enumerate}[]
    \item[1/] We show that such a price has the Markovian representation~$u(t,\Wh\otimes_t\Theta^t)$, for a map $u:[0,T]\times C^0([0,T],\RR^d)\to\RR$; see Proposition~\ref{prop:rpz};
    \item[2/] Under the additional assumption that~$\phi$ and~$f_t(\cdot)$ are three times differentiable with suitable growth conditions (which include the square root for~$\phi$), we prove that the map~$u$ is of class~$C^{1,2}$ with appropriate regularity and growth estimates. We further establish~$u$ as the unique classical solution to a path-dependent PDE in this class; see Theorem~\ref{thm:main}.
    \item[3/] We observe that the total implied variance (equivalently the implied volatility) satisfies a quasi-linear path-dependent PDE; see Theorem~\ref{th:IV}.
\end{enumerate}
In the case of VIX options, our first contribution shows that the payoff is not simply path-dependent, as a function of~$\Wh_{[0,T]}$, but involves the trajectory up to a time~$\Tb>T$ as well as a conditional expectation with respect to~$\Ff_T$, which perturbs the law of the underlying. In addition to the process~$\Theta$ indexed by two time variables, it thus requires the introduction of a process~$J$ indexed by three time variables which plays a central role in the analysis of the PPDE in the proof of Theorem~\ref{thm:main}.

This second contribution relies on the functional Itô formula~\cite[Theorem 3.17]{viens2019martingale}, for which we require fine estimates on the Fréchet derivatives of~$u$ (see Definition~\ref{def:Cplusalpha}). In order to present the main equation of this paper, we introduce a few definitions and refer to Section~\ref{sec:derivatives} for more details. We use the shorthand notation~$K^t:=K(\cdot,t)$ and denote its continuous truncation~$K^{\delta,t}:s\to K(s\vee(t+\delta),t)$. This allows to define~$\langle D_{\bx}^2 u(t,\bx),(K^{\delta,t},K^{\delta,t})\rangle$ as the second-order Fréchet derivative in the space~$C^0([t,T],\RR)$, and~$\big\langle \partial^2_{\bx} u(t,\bx), (K^t,K^t) \big\rangle$ as its limit when~$\delta$ tends to zero. The conditions met by $u$ and its derivatives depend on the singularity of the kernel~$K$ and require enough smoothness for this approximation to converge. In the course of proving Theorem~\ref{thm:main}, we identify additional estimates, absent from~\cite{viens2019martingale}, which seem necessary for obtaining the temporal derivative~$\partial_t u$. These estimates for the Fréchet derivatives These estimates for the second Fréchet derivative (presented in Lemma \ref{lemma:DerTwo}) involve the $L^2$--norm of the direction, in contrast to the~$L^\infty$--norm demanded in Definition~\ref{def:Cplusalpha}. Applying the functional Itô formula to the newly introduced process~$J$, we then provide a self-contained proof that~$u$ is the unique solution of the Cauchy problem 
\begin{equation*}
    \left\{\hspace{-.4cm}
\begin{array}{rl}
    &\displaystyle \partial_t u(t,\bx) + \half \big\langle \partial^2_{\bx} u(t,\bx), (K^t,K^t) \big\rangle = 0, \quad \text{for all   } (t,\bx)\in [0,T]\times C^0([0,\Tb],\RR^d),\\
    &\displaystyle  u(T,\bx) = \phi\circ\frF(\bx),
\end{array}
\right.
\end{equation*}
where~$\mathfrak{F}$ has an integral representation over~$[0,\Tb]\times\RR^d$ and which definition is given in Proposition~\ref{prop:rpz}.
 

This infinite-dimensional analogue of the heat equation has the advantage of being linear parabolic, with an explicit terminal condition.
Much effort has been dedicated to faster pricing of VIX options and even greater to the simulation of fractional processes but, meanwhile, PPDE schemes are still in their infancy (the literature consists in three papers~\cite{jacquier2023deep,sabate2020solving,saporito2021path} which all rely on inserting a discretisation of the path in a neural network).
Provided the path is encoded in an efficient way, this alternative viewpoint serves as the theoretical groundwork for new numerical schemes that spit out the whole implied volatility surface and the greeks without relying on simulation or Monte Carlo. Signature kernel methods~\cite{salvi2021signature} appear to be strong candidates as they combine a tailored encoding of the path through its signature, an efficient numerical computation and convergence guarantees. For those reasons we apply them to this problem in the recent work~\cite{pannier2024path}.

Implied volatilities are the unitless equivalent to option prices which enable practitioners to compare financial products with different features. It is obtained by inverting the Black-Scholes formula which is not an easily tractable technique. Our third contribution consists in retrieving two quasi-linear path-dependent PDEs satisfied by the total implied variance and the implied volatility.
This PDE representation was not known for volatility derivatives and is one of very few pre-asymptotic results in the literature.
The computation is inspired by~\cite{berestycki2004computing} where uniqueness and asymptotic expressions are also established.

Besides numerical schemes, PPDEs may therefore find other applications in asymptotic analysis and they were already instrumental in obtaining weak rates of convergence for rough volatility models in~\cite{bonesini2023rough}. Analogously to the finite-dimensional theory, the PPDEs offer additional tools for the analysis of complex volatility derivatives and their dynamics. On the more fundamental side, several open questions are worthy of interest, such as:
Can we relax the assumption~$\phi\in C^3$ to bare convexity, in particular to show that the Call price is in $C^{1,2}$? Is there a maximum principle of the type~$\sup_{(t,\bx)} u(t,\bx) = \sup_{\bx} u(T,\bx)$ as for the finite-dimensional heat equation? These are left for future research.

Finally, we observe in Section~\ref{sec:Markov} that, in Markovian models, if the volatility derivative does not depend on the trajectory on~$[0,T]$ (e.g. $\VIX_T$ which only acts on~$[T,T+\Delta]$) then neither does the option price. This leads to finite-dimensional PDEs displayed in Corollary~\ref{coro:Markov_PDE} and implemented for the pricing of VIX options in a two-factor Bergomi model where~$\Wh$ is an OU process. The rapidity and efficiency of this approach appoints it as an alternative to Monte Carlo methods, especially if one yearns for prices with respect to initial conditions, time to maturity, or for financial greeks.

\subsection{Organisation of the paper} The rest of the paper is arranged as follows. Section~\ref{sec:framework} introduces the model and some notations. In Section~\ref{sec:rpz} we develop and prove our first result, the Markovian representation of the option price. The path-dependent PDE is presented in Section~\ref{sec:PPDE}, followed by some additional results on the greeks, and preceded by the definitions of the pathwise derivatives and the  subspace of~$C^{1,2}$ of interest. We derive the PPDE for the implied volatility in Section~\ref{sec:IV}, discuss the Markovian case in Section~\ref{sec:Markov} and, finally, the proof of the main result is gathered in Section~\ref{sec:proof}.

\section{Framework}\label{sec:framework}
\subsection{The model}
Let us fix a filtered probability space $(\Omega,\Ff,(\Ff_t)_{t\ge0},\PP)$ as well as two finite times~$0<T\le \breve{T}$ and two intervals $[0,T]$ and $[0,\Tb]$. The conditional expectation with respect to the filtration will be denoted, for all $t\in[0,\Tb]$, as $\EE_t[\cdot]:=\EE[\cdot \lvert \Ff_t]$. Let~$m,d\in\NN$, let $K:\RR_+^2\to \RR^{d\times m}$ be a square-integrable matrix-valued kernel such that $K(t,r)=0$ for $t<r$ and $\norm{K}_2:=\sup_{t\in[0,\Tb]}\int_0^t \abs{K(t,r)}^2\dr $ is finite, and let~$W$ be a standard $m$-dimensional Brownian motion with respect to~$(\Ff_t)_{t\ge0}$.
We introduce the Gaussian Volterra process 
\begin{equation}\label{eq:Def_Wh}
    \Wh_t := \gamma_t+ \int_0^t K(t,r) \D W_r, \quad t\in[0,\Tb],
\end{equation}
with mean $(\gamma_t)_{t\in[0,\Tb]}$, a deterministic continuous path, and covariance structure
$$
\EE\left[ \big(\Wh_t-\gamma_t\big)\big( \Wh_s-\gamma_s\big)^\top\right] = \int_0^{t\wedge s} K(t,r) K(s,r)^\top \dr, \quad s,t\in[0,\Tb].
$$
We make the standing assumption that~$\Wh$ is a continuous stochastic process. This condition is
directly linked to the behaviour around zero of the function~$\mathcal{K}(\tau):=\sup_{s,t\in[0,\Tb], \abs{t-s} \le\tau} \int_0^t \abs{K(t,r)-K(s,r)}^2 \dr$ and can be checked easily using Kolmogorov's or Fernique's continuity criterion, see~\cite[Lemma 2.3]{gulisashvili2021time} for the latter. All the kernels presented in Example~\ref{ex:kernels} do satisfy this condition.

Definition~\eqref{eq:Def_Wh} is slightly more general than what is usually found in the literature where~$\gamma\equiv0$. In this case, the Gaussian process has memory of the past until an arbitrary fixed time~$t=0$. However, when pricing an option at time~$t=0$ (the moment the option is issued), one may also take into account prior information; this is precisely what the function~$\gamma$ contains. 
This paper only considers deterministic paths~$\gamma$, yet several models were proposed in the literature featuring~$\Ff_0$-measurable initial curves~$\gamma$. Two noteworthy instances are the original fractional Brownian motion of Mandelbrot and Van Ness~\cite{mandelbrot1968fractional} and the Brownian semistationary processes~\cite{bennedsen2022decoupling}.

The variance derivatives we will look at are encompassed by this general definition:
\begin{align}
    V_T := \EE_T \int_{0}^{\Tb} \sigma^2_s  \ds, \quad \text{where}\quad   \sigma^2_s := f_s(\Wh_s) \quad \text{for all  } s\in[0,\Tb],
\end{align}
and $f:\RR_+\times  C^0 ([0,\Tb], \mathbb{R}^d)\to \RR_+$ is a continuous function. This includes forward variance models as introduced by Bergomi~\cite{bergomi2008smile2}, see also Example~\ref{ex:volfct}. 
$$
\Dd:= \overline{\{t\in[0,\Tb]\,\lvert \exists x\in\RR^d\text{   such that   } f_t(x)\neq0\}}.
    $$
The following example illustrates this.  
\begin{example}[Derivatives]\label{ex:derivatives}
    We consider a continuous time stochastic volatility model with no interest rate defined by the SDE~$\D S_t/S_t = \sigma_t\D B_t$.
    \begin{itemize}
        \item Letting $\Tb=T+\Delta$ (with $\Delta=30/365$ corresponding to one month) and $f_s(x)=0$ for $s<T$ (hence~$\Dd=[T,T+\Delta]$) yields $V_T=\Delta^2 \VIX^2_T$. For~$\phi:\RR_+\to\RR$, the conditional expectation~$\EE_t[\phi(V_T)]$ corresponds to the price at time~$t\in[0,T]$ of a VIX future if~$\phi(x)=\frac{\sqrt{x}}{\Delta}$ and a VIX Call if~$\phi(x)=(\frac{\sqrt{x}}{\Delta}-K)_+$ with~$K>0$.
        \item If one sets $\Tb=T$ then $V_T/T$ is the realised variance and~$\Dd=[0,T]$. For~$K>0$, the random variable~$\EE_t[\phi(V_T)]$ corresponds to the price at time~$t\in[0,T]$ of a variance swap if~$\phi(x)=\frac{x}{T}-K$ and a Call on realised variance if~$\phi(x)=(\frac{x}{T}-K)_+$ 
    \end{itemize} 
\end{example}
We can also give examples of popular models covered by our setup.
\begin{example}[Kernels]\label{ex:kernels}
The following types of one-dimensional kernels account for most of the models found in the literature. One can build matrix-valued kernels by taking those as entries.
For~$c\in\RR$:
    \begin{itemize}
        \item \emph{Exponential:} $K(t,r)=c\E^{-\beta (t-r)}$, for~$\beta>0$, means that~$\Wh$ is an Ornstein-Uhlenbeck process as chosen in~\cite{bergomi2008smile2,abi2025joint}.
         \item \emph{Power-law:} $K(t,r)=c(t-r)^{H-\half}$ corresponds to the class of rough volatility models for~$H\in(0,1/2)$~\cite{bayer2016pricing} and long-memory fractional volatility models for~$H\in(1/2,1)$~\cite{comte1996long}. Here $\mathcal{K}(\tau) \le C \tau^{2H}$ hence Kolmogorov's continuity theorem implies that the trajectories of (a modification of)~$\Wh$ are~$\alpha$-Hölder continuous for all~$\alpha<H$. 
        \item \emph{Shifted power-law:} $K(t,r)=c(t+\ep-r)^{H-\half}$ yields a semimartingale path-dependent model whenever~$\ep>0$ and extends the range of the Hurst exponent to~$(-\infty,\half]$. We refer to~\cite{abi2025joint} for more details.
        \item \emph{Gamma:} $K(t,r)=c \E^{-\beta (t-r)} (t-r)^{H-\half}$, for~$\beta>0$ and~$H\in(0,1/2)$, leads to a rough volatility model with an exponential damping and the same continuity properties as above. We refer for instance to~\cite{bennedsen2017hybrid} and the references therein.
        \item \emph{FBM:} $K(t,r)=c\big[(t/r)^{H-\half}(t-r)^{H-\half}-(H-1/2)s^{\half-H}\int_s^t u^{H-3/2}(u-s)^{H-\half}\du\big]$ recovers the fractional Brownian motion of Mandelbrot and Van Ness~\cite{mandelbrot1968fractional} in the case~$H\in(0,1/2)$, which also has $\alpha$-Hölder continuous paths for all~$\alpha<H$. A proof of this representation can be found in~\cite[Proposition 5.1.3]{nualart2006malliavin}.
        \item \emph{Log-fBM:} $K(t,r)=c(t-r)^{H-\half} \max(\zeta\log((t-r)^{-1}),1)^{-p}$, where~$H\in[0,\half)$, $\zeta>0$ and~$p>1$. For~$H>0$, the associated process has the same continuity property as in the power-law case, however this log-modulation includes the case~$H=0$ which also induces a continuous process~\cite{bayer2021log}.
    \end{itemize}
When the power-law is present, the constant is usually set to~$c=\Gamma(H+\half)^{-1}$.
\end{example}

\begin{example}[Volatility functions]\label{ex:volfct}
    Different choices of non-linearity have been proposed, as the introduction witnesses. All of them can be multiplied by an initial curve~$\zeta:\RR_+\to\RR_+$ to match the forward variance curve. Let~$\Ee:L^2(\Omega,\RR)\to\RR$ denote the Wick stochastic exponential~$\Ee(X):=\exp(X-\half\EE[X^2])$.
    \begin{itemize}
        \item One-factor (rough) Bergomi model: $f_t(\Wh_t) = \zeta(t) \Ee(\nu\Wh_t)$, for~$\nu\in\RR$.
        \item Multi-factor (rough) Bergomi model: 
        \begin{equation}\label{eq:2FRB}
             f_t(\Wh_t) = \zeta(t) \sum_{i=1}^d \lambda_i \Ee\Big(\Wh^i_t\Big),
        \end{equation}
        where~$\lambda\in\RR_+^d$ and $\sum_{i=1}^d \lambda_i=1$ so that~$f(0,0)=\zeta(0)$. Note that, for~$i\neq j$, $\Wh^i$ and~$\Wh^j$ are correlated if the kernel matrix~$K$ is not diagonal. 
        \item A variety of models are studied by R{\o}mer~\cite{romer2022empirical}, including the two-factor rough Bergomi model, the two-factor hyperbolic model which consists in replacing~$\Ee$ by~$H(x)=x+\sqrt{x^2+1}$ in~\eqref{eq:2FRB}, and a two-factor model that uses both hyperbolic and quadratic transformations. The details of the latter are shown in~\cite[Eqs (45)-(48)]{romer2022empirical}.
        \item Quintic Ornstein-Uhlenbeck model~\cite{abi2025joint}: 
        $$
f_t(\Wh_t) = \zeta(t) \frac{p(\Wh_t)^2}{\EE[p(\Wh_t)^2]}, \quad p(x) = \alpha_0 + \alpha_1 x + \alpha_3 x^3 + \alpha_5 x^5, 
        $$
        where~$\alpha_0,\alpha_1,\alpha_3,\alpha_5\ge0$. In this model, $\Wh$ is an OU process with fast-mean reversion and high vol-of-vol, i.e. $\beta$ and $c$ are large in the first kernel of Example~\ref{ex:kernels}.
    \end{itemize}
\end{example}

\subsection{Notations}
The notation $\abs{\cdot}$ corresponds to the Euclidean norm in~$\RR^d$ and Frobenius norm in~$\RR^{d\times d}$. Let $C^0$ and $D^0$ be the spaces of continuous and càdlàg paths, respectively. 
We introduce the following notations:
\begin{align*}
    &\norm{\bx}_{\infty} = \sup_{t \in [0,\Tb]} \abs{\bx_s}, \quad {\rm d}((t,\bx),(\tb,\bxb)):= \abs{t-\tb} + \norm{\bx-\bxb}_{\infty},\quad 
    \Lambda := [0,T] \times C^0 ([0,\Tb], \mathbb{R}^d),\\
    &\widetilde{\Lambda} := \Big\{(t, \bx)\in[0,T] \times D^0 ([0,\Tb], \mathbb{R}^d):\bx\!\mid_{[t, \Tb]}\in C^0 ([t,\Tb], \mathbb{R}^d) \Big\},\\
    &\Ww_t := \Big\{\bx \in D^0 ([0,\Tb], \mathbb{R}^d) : \bx \lvert_{[0, t)} \, = 0 \text{ and } \bx\lvert_{[t,T]}\in C^0([t,T],\RR^d) \Big\}.
\end{align*}
Let $C^0(\widetilde{\Lambda}):=C^0(\widetilde{\Lambda}, \RR)$ denote the set of functions~$u:\widetilde{\Lambda}\to\RR$ continuous under ${\rm d}$. 
In the remainder, the time horizon $T$ will correspond to the maturity of the derivatives, which are then stochastic processes on~$[0,T]$. Meanwhile, the underlying Gaussian process is defined on $[0,\Tb]$, hence its trajectories live in $ C^0 ([0,\Tb], \mathbb{R}^d)$. 
For $(t,\bx)\in\Lambda$, the direction~$\eta$ of the Fréchet derivative will belong to~$\Ww_t$ to ensure that the path~$\bx$ is only perturbed after time $t$, and hence the perturbed path~$\bx+\eta$ will live in~$\widetilde{\Lambda}$.

\section{The Markovian representation} \label{sec:rpz}
For a function $\phi:\RR\to\RR$, the main protagonist of this paper is the stochastic process $\big\{\EE_t[\phi(V_T)] : t\in[0,T]\big\}$ which represents the price of the volatility derivative under a risk-neutral measure.
For all $0\le t<s$, one of the central ideas behind the functional Itô formula of \cite{viens2019martingale} is to split $\Wh$ in two integrals:
\begin{equation}\label{eq:Decomposition_Theta}
\Wh_s = \left(\gamma_s+\int_0^t  K(s,r)\D W_r\right) +\int_t^s K(s,r)\D W_r  =: \Theta^t_s + I^t_s,
\end{equation}
where $\Theta^t$ is an $\Ff_t$-measurable process while $I^t_s$ is independent of $\Ff_t$ for all $s>t$. This orthogonal decomposition into two processes with two time indexes arises naturally when considering conditional expectations.
It is clear that in the Brownian motion case $K\equiv1$, this decomposition boils down to $W_s=W_t +(W_s-W_t)$.
For $t\in[0,T]$ and two paths~ $\bx,\theta\in C(\RR_+,\RR^d)$, their concatenated path reads $(\bx\otimes_t\theta)_s := \bx_s \one_{t>s} + \theta_s \one_{t\le s}$. For $t\in[0,T]$, we further introduce the process
$$
\Th^t := \Wh\otimes_t \Theta^t = \EE_t[\Wh_\cdot],
$$
and for all~$s\in[0,\Tb]$, we notice that~$\Th^t_s=\int_0^{t\wedge s}K(s,r)\D W_r=\int_0^{t}K(s,r)\D W_r$ since~$K(s,r)=0$ for~$s<r$. Therefore~$\Th$ is an extension of~$\Theta$ to the domain~$[0,\Tb]$. 
These notations would be sufficient were we to consider path-dependent payoffs~$\phi(\Wh)$ as in~\cite{viens2019martingale}. However, handling the conditional expectation in the definition of~$V_T$ itself requires the introduction of a process indexed by three time variables
\begin{align*}
    J^{t,T}_s:=
    \Th^T_s-\Th^t_s=
    \int_t^{s\wedge T} K(s,r)\D W_r=
    \int_t^{T} K(s,r)\D W_r.
\end{align*}
For $t>s$, note that~$J^{t,T}_s=0$ since~$K(s,r)=0$.  
These processes are at the core of our first main result.
\begin{proposition}\label{prop:rpz}
   Let~$\phi:\RR\to\RR$ and~$f:\Dd\times\RR\to\RR_+$ be measurable functions such that~$\EE[\abs{\phi(V_T)}]<\infty$. The measurable map $u:\Lambda\to\RR$ defined by
       \begin{equation}\label{eq:Def_u}
        u\big(t, \bx\big)
       : =\EE\Big[\phi\circ\frF\Big(
    \big\{ \bx_s + J^{t,T}_s :s\in\Dd\big\}\Big)\Big],
    \end{equation}
    where
    \begin{equation}\label{eq:Def_frF}
        \frF(\bx) := \int_0^T f_s(\bx_s)\ds + \int_T^{\Tb} \EE\Big[f_s\big(\bx_s+I^T_s\big)\Big]\ds,
    \end{equation}
    allows for the following representation of the conditional expectation
    $$
u\Big(t,\Big\{ \Th^t_s  : s\in[0,\Tb]\Big\}\Big)=\EE_t\big[\phi(V_T)\big] \qquad \text{for all   } t\in[0,T].
    $$
This map will be referred to as the option price or the value function.
\end{proposition}
\begin{remark}
    Sufficient conditions for $\EE[\abs{\phi(V_T)}]<\infty$ to hold will be given in Section~\ref{sec:main}.
\end{remark}
\begin{remark}
In the VIX case we have $\Th^{t}_s=\Theta^t_s$ for all $s\in\Dd=[T,\Tb]$, in other words the option price does not depend on the past $\{\Wh_s:s\in[0,T]\}$, but depends on a path that runs after the maturity $T$.
On the other hand, the RV option corresponds to~$\Dd=[0,T]$ where~$J^{t,T}_s=\int_t^s K(s,r)\D W_r$.  
\end{remark}
\begin{proof}
We recall that 
\begin{align*}
    \EE_t\big[\phi(V_T)\big]= \EE_t\left[\phi\left(\int_{0}^{\Tb} \EE_T[\sigma^2_s]  \ds\right)\right],
\end{align*}
which features two convoluted conditionings at $t$ and $T$.

We first take a look at $V_T$.
For $s<T$, $\EE_T [\sigma^2_s]=\sigma^2_s$, while for $s\ge T$, 
$\sigma^2_s =f_s(\Theta^T_s +I^T_s)$ as in \eqref{eq:Decomposition_Theta}.
Since $I^T_s$ is independent of $\Ff_T$ and $\Theta^T_s$ is $\Ff_T$-measurable, there exist two measurable functions~$\frf:\RR_+\times\RR\to\RR$ and $\frF: C^0 ([0,\Tb], \mathbb{R}^d)\to\RR$ such that
\begin{align}
    V_T=\EE_T \int_0^{\Tb} \sigma^2_s\ds
    =\int_0^{T} f_s(\Wh_s)\ds + \int_T^{\Tb} \EE_T\left[ f_s\big( \Theta^T_s + I^T_s \big)\right]\ds
    &=: \int_0^{T} f_s(\Wh_s)\ds + \int_T^{\Tb} \frf_s(\Theta^T_{s})\ds \nonumber\\
    &=\mathfrak{F} \Big( \Big\{\Th_s^T : s\in \Dd\Big\} \Big). 
    \label{eq:VT_frF}
\end{align}
Let $s\in\Dd$ and $\theta\in\RR^d$, then we can define explicitly $\frf_s(\theta)=\EE[f_s(\theta+I^T_s)]$. Note that $I^T_s$ has a Gaussian distribution with zero mean and variance $\displaystyle \Sigma_s:=\int_T^s K(s,r)K(s,r)^\top \dr$. Hence, for $\theta\in \RR^d$,
\begin{equation}\label{eq:fstheta}
\mathfrak{f}_s(\theta) = \int_{\RR^d}  f_s(\theta+y) p_s(y) \D y,
\end{equation}
where $p_s:\RR^d\to\RR$ is the density function of $\mathcal{N}\left(\bm{0}_d, \Sigma_s\right)$ if~$\Sigma_s$ is non-degenerate.
Furthermore, we clearly have
\begin{equation*}
\frF(\bx) = \int_0^T f_s(\bx_s)\ds + \int_T^{\Tb} \frf_s(\bx_s)\ds.
\end{equation*}
We turn to the second conditioning with respect to $t\in[0,T]$. 
We decompose the path $\Wh\otimes_T\Theta^T$, for all~$s\in [0,\Tb]$:
\begin{align}
    \Th^T_s 
    &= \int_0^{T\wedge s} K(s,r)\D W_r
    = \int_0^{t\wedge s} K(s,r)\D W_r + \int_t^{T\wedge s} K(s,r)\D W_r 
    = \Th^t_s + J^{t,T}_s, \label{eq:def_JtT}
\end{align}
where we recall that~$J^{t,T}_s=\int_t^{T} K(s,r)\D W_r$. Introduce the map~$u:\Lambda\to\RR$ by
\begin{equation*}
        u\big(t, \{ \bx_s: s\in[0,\Tb]\} \big)
       : =\EE\Big[\phi\circ\frF\Big(
    \big\{ \bx_s + J^{t,T}_s :s\in\Dd\big\}\Big)\Big].
    \end{equation*}
Since $J^{t,T}_{s}$ is independent of $\Ff_t$, we can write options on $V_T$ under our stochastic volatility model as 
\begin{align}\label{eq:uMarkov}
    u\Big(t,\Big\{ \Th^t_s  :s\in[0,\Tb]\Big\}\Big)
    &=  \EE\bigg[\phi\circ\frF\Big(
    \Big\{ \Th^t_s + J^{t,T}_s :s\in\Dd\Big\}\Big)\Big\lvert \Th^t\bigg] \\
    &= \EE\Big[\phi\circ \mathfrak{F} \Big( \Big\{\Th^T_s : s\in \Dd\Big\} \Big) \Big\lvert \Ff_t \Big]  \nonumber\\
    &= \EE_t\Big [\phi(V_T)\Big],   \nonumber
\end{align}
almost surely, where we expressed $V_T$ as in \eqref{eq:VT_frF}. 
\end{proof}
\begin{remark}
In the VIX case ($\Dd=[T,\Tb]$) we have $J^{t,T}=\Theta^T - \Theta^t$ and in the RV case ($\Dd=[0,T]$) we have $J^{t,T}=I^t$.    
\end{remark}
In addition to Proposition \ref{prop:rpz}, one can give an insightful peak on the path-dependence structure of~$V_T$.
For $t\in[0,T], s\in[t,\Tb]$ and $\bx\in  C^0 ([0,\Tb], \mathbb{R}^d)$, let us define
\begin{equation}\label{eq:defs_tx}
    \Wh_s^{t,\bx} := \bx_s + \int_t^s K(s,r)\D W_r \quad \text{and} \quad
    V^{t,\bx}_T :=  \EE_T \int_{0}^{\Tb} f_s\big(\Wh_s^{t,\bx}\big) \ds.
\end{equation}
 Notice that, for all~$t\in[0,\Tb]$, $\Wh^{t,\Th^t}=\Wh$ and therefore $V_T^{t,\Th^t}=V_T$. By fixing the trajectory $\Th^t=\bx$ for some $\bx\in C^0 ([0,\Tb], \mathbb{R}^d)$, Equation \eqref{eq:Def_u} yields 
    \begin{equation}\label{eq:Vtomega}
        u\big(t, \{ \bx_s: s\in[0,\Tb]\} \big)
        =\EE\Big[ \phi\big(V_T^{t,\bx}\big) \Big].
    \end{equation}
The role of the shift~$\gamma$ is made more precise here. We observe that~$\Wh=\Wh^{0,\gamma}$ hence~$\EE[\phi(V_T)]=\EE[\phi(V_T^{0,\gamma})]=u(0,\gamma)$. Pricing at time~$t=0$ without including~$\gamma$ in the model comes down to computing~$u(0,0)$ and ignoring all the information that occured before~$t=0$. Introducing this shift provides a more coherent theory across all times.
    
The representation~\eqref{eq:Vtomega} allows to describe a type of time invariance for the option price. Let us write~$u(t,\bx;T,\Dd)=u(t,\bx)$ to highlight the dependence in the terminal time and the time interval. It is common for solutions to PDEs on~$[0,T]\times\RR^d$ to have the property~$u(t,x;T)=u(0,x;T-t)$.  A similar phenomenon takes place in our setting under the assumption that~$K$ is of convolution type.
\begin{corollary}\label{coro:time_invariance}
    Let~$0 \le T \le \Tb$ and assume there exists~$\tilde{K}\in L^2([0,\Tb],\RR^{d\times m})$ such that~$K(s,t)=\tilde{K}(s-t)$ for all~$0\le t\le s\le \Tb$.
    Then for all~$\bx\in C^0([0,\Tb],\RR^d)$ and measurable functions~$\phi:\RR\to\RR$ and~$f:\Dd\times\RR^d\to\RR$, we have
    $$
u(t,\bx;T,\Dd) = u(0,\bx_{t+\bullet};T-t,\Dd-t),
    $$
    where~$\Dd\subset[0,\Tb]$ is the support of~$f_\bullet$ and~$s\in\Dd-t$ if~$s+t\in\Dd$.
\end{corollary}
\begin{remark}
    For both VIX and Realised Variance, $\Dd-t$ corresponds precisely to the interval of interest for an option of maturity~$T-t$. 
In financial terms, this invariance property means one can translate the price at time~$t$ of an option of maturity~$T$  to the price at time~$0$ of an option of maturity~$T-t$, provided one shifts the path accordingly. 
\end{remark}
\begin{proof}
    Let~$0\le t \le T \le \Tb$ and $\bx\in C^0([0,\Tb],\RR^d)$.
    We first show that~$J^{t,T}_s=J^{0,T-t}_{s-t}$ in distribution:
    \begin{align*}
        J^{t,T}_s = \int_t^{T} \tilde{K}(s-r) \D W_r 
        &=  \int_0^{T-t} \tilde{K}(s-t-r) \D W_{r+t} \\
        &\overset{\mathrm{law}}{=}  \int_0^{T-t} \tilde{K}(s-t-r) \D W_{r} 
        = J^{0,T-t}_{s-t}.
    \end{align*}
    Equation~\eqref{eq:Def_u} entails
    \begin{align*}
        u(t,\bx;T,\Dd) 
        = \EE\left[\phi\circ\frF\left(\big\{\bx_s+J^{t,T}_s:s\in\Dd\big\}\right)\right]
        &=\EE\left[\phi\circ\frF\left(\big\{\bx_s+J^{0,T-t}_{s-t}:s\in\Dd\big\}\right)\right]\\
        &=\EE\left[\phi\circ\frF\left(\big\{\bx_{s+t}+J^{0,T-t}_s:s+t\in\Dd\big\}\right)\right],
    \end{align*}
    which yields the claim.
\end{proof}


\section{The path-dependent pricing PDE}\label{sec:PPDE}
Proposition \ref{prop:rpz} confirms the natural intuition that options on variance, whether realised or VIX, are functions of a path. Thus the associated pricing PDE is of path-dependent type as we will see in this section. 

\subsection{The pathwise derivatives}
\label{sec:derivatives} 
Let us define the right time derivative
$$
\partial_t u(t,\bx):= \lim_{\ep\downarrow0}\frac{u(t+\ep,\bx)-u(t,\bx)}{\ep},
$$
for all~$(t,\bx)\in\widetilde{\Lambda}$, provided the limit exists. We also define the Fréchet derivative~$D_{\bx} u(t,\bx)$ with respect to~$\bx\one_{[t,T]}$, which is a linear operator on ~$ \Ww_t$: 
\begin{equation}
    u(t,\bx+\eta)-u(t,\bx) = \langle D_{\bx} u(t,\bx),\eta\rangle
    +o(\norminfty{\eta}), \quad \text{for all   } \eta\in \Ww_t.
    \label{eq:PathwiseDerDef}
\end{equation}
We recall that any $\eta\in\Ww_t$ is equal to zero on $[0,t)$ hence the derivative in this particular direction only perturbs the path on~$[t,T]$. 
If it exists, it is equal to the Gateaux derivative
\begin{equation}
\lim_{\ep\downarrow0}\frac{u(t,\bx+\ep\eta)-u(t,\bx)}{\ep}.
\label{eq:PathwiseDerDef2}
\end{equation}
The second derivative~$D_{\bx}^2 u(t,\bx)$, defined for all $\eta,\etab\in \Ww_t$ as
\begin{align*}
    \langle D_{\bx}u(t,\bx+\eta),\etab\rangle
    -\langle D_{\bx}u(t,\bx),\etab\rangle
    = \langle D^2_{\bx} u(t,\bx),(\eta,\etab)\rangle + o(\norm{\eta}_{\infty}),
\end{align*}
is a bilinear operator on $ \Ww_t\times  \Ww_t$. We say that~$u\in \Cc^{1,2}(\widetilde{\Lambda})$ if $\partial_t u$, $D_\bx u$ and $D^2_\bx u$ exist and are continuous on $\widetilde{\Lambda}$, that is: for all $\eta\in C^0 ([0,\Tb], \mathbb{R}^d)$, $\widetilde{\Lambda}\ni (t,\bx)\mapsto  \partial u(t,\bx)$ is continuous under $d$ for all~$\partial\in\big\{\partial_t, \langle D_\bx \bullet,\eta\rangle, \langle D_\bx^2 \bullet,(\eta,\etab)\rangle\big\}.$

If $\eta$ is an $\RR^{d\times m}$-valued path (like~$K^t$) with $\eta^{(j)}$ being the ($\RR^d$-valued) $j$th column then, by a mild abuse of notation, we define
\begin{align*}
    \big\langle D_{\bx} u(t,\bx),\eta \big\rangle
    := \Big(\big\langle D_{\bx} u(t,\bx),\eta^{(j)}\big\rangle\Big)_{j=1}^m,\quad \text{and}\quad
    \big\langle D_{\bx}^2 u(t,\bx),(\eta,\etab)\big\rangle 
    := \sum_{j=1}^m \big\langle D_{\bx}^2 u(t,\bx),(\eta^{(j)},\etab^{(j)})\big\rangle.
\end{align*}
The solution space is adapted to the singularity of the kernel $K$, hence we start with describing the latter.
\begin{assumption}\label{assu:kernel}
    For any $0\le t< s$, $\partial_s K(s,t)$ exists and there exist $C>0$ and $H\in(0,1)$ such that
    $$
\abs{K(s,t)}\le C (s-t)^{H-\half} \qquad \text{and} \qquad \abs{\partial_s K(s,t)}\le C (s-t)^{H-3/2}.
    $$
\end{assumption}
This assumption includes most kernels found in the literature, in particular those presented in Example~\ref{ex:kernels}, moreover it does not impose any condition on the structure but only on the speed of the explosion in the diagonal.
\begin{definition}\label{def:Cplusalpha}
    We say that $u\in \Cc^{1,2}_{\alpha}(\Lambda)$, with $\alpha\in(0,1)$, if 
    there exists an extension of~$u$ in~$\Cc^{1,2}(\widetilde{\Lambda})$, still denoted as~$u$, a growth order~$\kappa>0$ and a modulus of continuity function~$\varrho$ such that, for any~$ t\in[0,T], \,0<\delta\le \Tb-t$, and~$\eta,\etab\in \Ww_t$ with supports contained in~$[t,t+\delta]$, the following hold:
    \begin{enumerate}
        \item[(i)] for any~$\bx\in D^0 ([0,\Tb], \mathbb{R}^d)$ such that~$\bx\one_{[t,\Tb]}\in C^0 ([t,\Tb], \mathbb{R}^d)$,
        \begin{align}
        \label{eq:DerOne_bound}
            \abs{\langle  D_{\bx} u(t,\bx),\eta\rangle} &\le C(1+\E^{\kappa \norm{\bx}_\infty}) \norm{\eta}_\infty\delta^\alpha,\\
        \label{eq:DerTwo_bound}
            \abs{\langle  D_{\bx}^2 u(t,\bx),(\eta,\etab)\rangle} &\le C(1+\E^{\kappa \norm{\bx}_\infty}) \norm{\eta}_\infty\norm{\etab}_\infty\delta^{2\alpha}.
        \end{align}
        \item[(ii)] for any other~$\bxb\in D^0 ([0,\Tb], \mathbb{R}^d)$ such that~$\bxb\one_{[t,T]}\in C^0 ([t,\Tb], \mathbb{R}^d)$,
        \begin{align}
        \label{eq:DerOne_reg}
            \abs{\langle  D_{\bx} u(t,\bx)- D_{\bx} u(t,\bxb),\eta\rangle} &\le C\Big(1+\E^{\kappa \norm{\bx}_\infty}+\E^{\kappa \norm{\bxb}_\infty} \Big) \norm{\eta}_\infty\varrho(\norm{\bx-\bxb}_\infty) \delta^\alpha,\\
        \label{eq:DerTwo_reg}
            \abs{\langle  D_{\bx}^2 u(t,\bx)- D_{\bx}^2 u(t,\bxb),(\eta,\etab)\rangle} &\le C\Big(1+\E^{\kappa \norm{\bx}_\infty}+\E^{\kappa \norm{\bxb}_\infty} \Big) \norm{\eta}_\infty\norm{\etab}_\infty\varrho(\norm{\bx-\bxb}_\infty)\delta^{2\alpha}.
        \end{align}
        \item[(iii)] For any~$\bx\in C^0 ([0,\Tb], \mathbb{R}^d)$, $t\mapsto\langle  D_{\bx} u(t,\bx),\eta\rangle$ and $t\mapsto\langle  D_{\bx}^2 u(t,\bx),(\eta,\etab)\rangle$ are continuous.
    \end{enumerate}
\end{definition}
This definition is an adaptation of~\cite[Definition 2.4]{viens2019martingale} where polynomial growth in~$\bx,\bxb$ was imposed.
The introduction of the parameter $\alpha$ touches upon a technical specificity of the singular kernel case. Setting~$\alpha>1/2-H$ ensures that the decaying factors $\delta^\alpha$ and~$\delta^{2\alpha}$ on the right hand side balances the explosion of $K$ quantified in Assumption \ref{assu:kernel}. 
Notice that in the regular case $H\ge1/2$ one can choose $\alpha=0$; in the rough case $H\in(0,1/2)$, we will show that the value functional $u$ defined in \eqref{eq:Def_u} belongs to $\Cc^{1,2}_{\alpha}(\Lambda)$ with $\alpha=1/2$.  

We then extend the domain of the Fréchet derivative via an approximation near the diagonal.
For~$0\le t <s \le \Tb$, we introduce the truncated kernel
$$
K^{\delta}(s,t):= K(s\vee(t+\delta),t),
$$
and the notations $K^t(s):=K(s,t)$ and~$K^{\delta,t}(s):= K^\delta (s,t)$.
For $u\in \Cc^{1,2}_{\alpha}(\Lambda)$, the spatial derivatives are defined as limits of Fréchet derivatives~\cite[Theorem 3.17]{viens2019martingale}
\begin{align}\label{eq:Def_SingularDerivatives}
    \langle \partial_{\bx} u(t,\bx),K^t\rangle &:= \lim_{\delta\downarrow0} \langle D_{\bx} u(t,\bx),K^{\delta,t}\rangle,\\ 
    \langle \partial_{\bx}^2 u(t,\bx),(K^{t},K^{t})\rangle &:= \lim_{\delta\downarrow0} \langle D_{\bx}^2 u(t,\bx),(K^{\delta,t},K^{\delta,t})\rangle.
\end{align}

\subsection{The main result}\label{sec:main}

This section presents the main result, namely a well-posed PPDE satisfied by the option price. In order to state it, we introduce the assumptions needed on the functions $\phi$ and $f$.
\begin{definition}\
\begin{enumerate}
    \item[(i)]  We say that a function $g:(0,+\infty)\to\RR$ has polynomial growth if there exist $C_g,\overline{\kappa_g},\underline{\kappa_g}\ge0$ such that $\abs{g(x)} \le C_g (1+\abs{x}^{\overline{\kappa_g}}+\abs{x}^{-\underline{\kappa_g}})$. For $n\in\NN$, we write $g\in C^n_{{\rm poly}}$ if $g$ is $n$ times continuously differentiable and $g$ and all its derivatives have polynomial growth.
    \item[(ii)] We say that a function $h:\RR_+\times\RR^{d}\to(0,+\infty)$ has exponential growth if there exist $C_h,\overline{\kappa_h}\ge0$ such that $\abs{h(t,x)} \le C_h \left(1+\E^{\overline{\kappa_h} t} \sum_{i=1}^{d} \E^{\overline{\kappa_h} x_i} \right)$.  For $n\in\NN$, we write $h\in C^n_{{\rm exp}+}$ if $h$ is $n$ times continuously differentiable and $h$ and all its derivatives have exponential growth.
    \item[(iii)] We say that a function $h:\RR_+\times\RR^{d}\to(0,+\infty)$ has exponential decay if there exist $c_h,\underline{\kappa_h}\ge0$ such that~$\abs{h(t,x)} \ge c_h \E^{-\underline{\kappa_h} t} \sum_{i=1}^{d} \E^{-\underline{\kappa_h} x_i}$.  For $n\in\NN$, we write $h\in C^n_{{\rm exp}-}$ if $h$ is $n$ times continuously differentiable and $h$ and all its derivatives have exponential decay. Moreover, we denote~$C^n_{{\rm exp}\pm}:=C^n_{{\rm exp}+} \cap C^n_{{\rm exp}-}$.
\end{enumerate}   
\end{definition}
\begin{assumption}\label{assu:coefs}\
\begin{enumerate}
    \item[(i)] The map $\phi$ belongs to $C^3_{{\rm poly}}(\RR)$, with the growth constants $C_\phi,\overline{\kphi},\underline{\kphi}$.
    \item[(iia)] If~$\underline{\kappa_\phi}=0$, the map $f$ belongs to $C^{0,3}_{{\rm exp}+}(\Dd\times\RR^d)$, with the growth constants $C_f,\kappa_f$.
    \item[(iib)] If~$\underline{\kappa_\phi}>0$, the map $f$ belongs to $C^{0,3}_{{\rm exp}\pm}(\Dd\times\RR^d)$, with the growth constants $C_f,c_f,\overline{\kappa_f},\underline{\kappa_f}$.
\end{enumerate}
\end{assumption}
\begin{remark}
Under the additional assumption of exponential decay, our setup includes payoffs of the type~$\phi(x)=x^p$ for any~$p\in\RR$. In particular, $\phi(x)=\sqrt{x}$ is crucial to study VIX futures and options. 
\end{remark}
\begin{theorem}\label{thm:main}
    Let Assumptions \ref{assu:kernel} and \ref{assu:coefs} hold. The value functional $u$ defined in \eqref{eq:Def_u} is the unique~$\Cc^{1,2}_{\half}(\Lambda)$ solution to  the path-dependent partial differential equation, for all~$(t,\bx)\in\Lambda$,
\begin{equation}\label{eq:main_PPDE}
\def\arraystretch{1.5}
    \left\{\hspace{-.4cm}
\begin{array}{rl}
    &\displaystyle \partial_t u(t,\bx) + \half \big\langle \partial^2_{\bx} u(t,\bx), (K^t,K^t) \big\rangle = 0,\\
    &\displaystyle  u(T,\bx) = \phi\circ\frF(\bx),
\end{array}
\right.
\end{equation}
where $\mathfrak{F}$ is defined in \eqref{eq:Def_frF} as
\begin{equation}\label{eq:Def_frF_bis}
    \frF(\bx) = \int_0^T f_s(\bx_s)\ds + \int_T^{\Tb} \EE\Big[f_s\big(\bx_s+I^T_s\big)\Big]\ds
    = \int_0^T f_s(\bx_s)\ds + \int_T^{\Tb} \int_{\RR^d} f_s\big(\bx_s+y\big)p_s(y)\D y\, \ds
    , 
\end{equation}
and $p_s:\RR^d\to\RR$ is the density function of $\mathcal{N}\left(\bm{0}_d, \int_T^s K(s,r)K(s,r)^\top \dr\right)$.
\end{theorem}
\begin{remark}\label{rem:mainPDE}
Several remarks are in order.
\begin{enumerate}
    \item This PDE is not homogeneous in time because of the direction $K^t$. However, in the case of convolution kernels, Corollary~\ref{coro:time_invariance} recovers a certain time invariance property.
    \item The derivative operator itself is of (path-dependent) heat type, already derived in \cite[Theorem 4.1]{viens2019martingale}, and common to all PPDEs linked to a Gaussian Volterra process. The novelty is the path-dependent nature of the terminal condition. It exhibits an extra layer, the function~$\frF$, which is a simple Riemann integral and can be learned offline. Indeed, functionals of a path are known to be well approximated by linear maps of their signature, in which case the learning phase does not rely on a fixed discretisation grid. One can thus evaluate the learned functional on a different time grid than the one it was trained on. 
    \item The vertical derivative of Dupire morally corresponds to \eqref{eq:PathwiseDerDef2} where~$\bx$ is frozen on~$[t,\Tb]$ and~$\eta$ is constant. Hence the PPDE \eqref{eq:main_PPDE} boils down to the functional PDE introduced in~\cite{dupire2019functional,peng2016bsde} in the Brownian motion case where $K\equiv I_d$. Rigorously speaking, the two notions of derivatives differ as they are not defined on the same spaces.
    \item The assumption that $\phi$ and $f$ are three times continuously differentiable is made for convenience but can be relaxed to two times differentiable with an~$\alpha$-Hölder continuous second derivative for any~$\alpha>0$.
    \item From a financial point of view, an important question remains. Say one solves the PPDE and knows the functional~$u$, then which path~$\bx$ should one use to compute the right price? 
If one is looking for the price at time~$t$ then one should input~$\bx=\Th^t$ to recover~$u(t,\Th^t)=\EE_t[\phi(V_T^{t,\Th^t})]=\EE_t[\phi(V_T)]$. Naturally, $\Th^t$ is not observed but it can be inferred from the forward variance, which is the derivative of the variance swap. Let us consider VIX derivatives in the rough Bergomi model, which means~$V_T=V_T^{t,\Theta^t}$ and~$f_s(\Wh_s)=\exp(\nu\Wh_s -\frac{\nu^2}{4H}s^{2H})$. The forward variance is given by
    \begin{equation*}
    \qquad\qquad\xi^t_s := \EE\left[f_s(\Wh_s) |\Ff_t \right] = e^{\nu\Theta^t_s - \frac{\nu^2}{4H}(s^{2H} - (s-t)^{2H})}, \quad \text{hence}\quad \Theta^t_s = \frac{\log(\xi^t_s)}{\nu} +\frac{\nu}{4H}(s^{2H} - (s-t)^{2H}).
    \end{equation*}
    In particular, at $t=0$, we have~$\Theta^0_s=\log(\xi^0_s)/\nu$ which is only attainable for a non-constant curve~$\xi^0$ if~$\Theta^0=\gamma$ is also not constant.
\end{enumerate}
\end{remark}

The proof of Theorem~\ref{thm:main} relies on the functional Itô formula of Viens and Zhang~\cite[Theorem 3.17]{viens2019martingale} which we state in our setup. For $\bx\in C^0 ([0,\Tb], \mathbb{R}^d)$ and all~$0 \leq t\le \Tb$ we consider the processes
\begin{align*}
    &\bm{X}_s := \bx_s +\int_t^s K(s,r)\D W_r, \quad\text{and}\quad
    \bm{\Theta}^{t+h}_s:=\bx_s+\int_t^{t+h} K(s,r)\D W_r.
\end{align*}
For $F\in \Cc^{1,2}_\alpha(\Lambda)$ with $\alpha>(\half-H)^+$ and $H\in(0,1)$, the singular Itô formula holds~\cite[Theorem 3.17]{viens2019martingale}:
\begin{align}
        \D F(r,\bm{X}\otimes_r \bm{\Theta}^r) =&\,  \partial_t F(r,\bm{X}\otimes_r \bm{\Theta}^r)\dr + \half \big\langle\partial_{\bm{x}}^2 F(r,\bm{X}\otimes_r \bm{\Theta}^r),(K^r,K^r )\big\rangle \dr \nonumber\\
        &+ \big\langle \partial_{\bm{x}} F(r,\bm{X}\otimes_r \bm{\Theta}^r),K^r \big\rangle \dr + \big\langle \partial_{\bm{x}} F(r,\bm{X}\otimes_r \bm{\Theta}^r), K^r \big\rangle \,\D W_r,
        \label{eq:Ito}
    \end{align}
    where we recall that~$\bm{X}\otimes_t \bm{\Theta}^t(s)=\bm{X}_s \one_{s\le t} + \bm{\Theta}^t_s \one_{s>t}$. Note that we extended the original definition of~$\Cc^{1,2}_\alpha$ to allow for exponential growth in Definition~\ref{def:Cplusalpha}. An inspection of the proof shows that the functional Itô formula holds true under the exponential growth of Definition~\ref{def:Cplusalpha} because $\EE[\E^{\kappa \norminfty{B}}]$ is finite for any Gaussian process $B$ and $\kappa>0$ \cite[Lemma 6.13]{jacquier2021rough}.
\begin{proof}[Proof of Theorem \ref{thm:main}]
We start by embedding the present framework into the SVE setting presented above.  The main computational part of the proof is postponed to Section~\ref{sec:proof} for clarity, and is summarised in the following lemma.
    
    \parbox{0.96\textwidth}{
    \begin{lemma}\label{lemma:Czerotwo}
            Let Assumptions \ref{assu:coefs} and \ref{assu:kernel} hold. The function $u$ defined in \eqref{eq:Def_u} belongs to~$\Cc^{0,2}_{\half}(\Lambda)$ with $\varrho={\rm id}$.
    \end{lemma}
    }

    \noindent 
    
    Thanks to the assumptions on~$\phi$ and $f$ and Lemma \ref{lemma:Estimates_fg}, we have that $\EE\big[\abs{\phi(V_T^{t,\bx})}^p\big]<\infty$ for all $(t,\bx)\in\Lambda$ and $p\ge1$. 
    Central to the proof is the orthogonal decomposition, for any~$t\le  t+h\le T$,
    \begin{align*}
        J^{t,T}_s = \int_t^{T} K(s,r)\D W_r 
        = \int_t^{t+h} K(s,r)\D W_r  + \int_{ t+h}^{T}K(s,r)\D W_r 
        =J^{t, t+h}_s+ J^{ t+h,T}_s.
    \end{align*}
    In light of the independence of~$J^{ t+h,T}_s$ with~$\Ff_{t+h}$, this grants us the representation of the value function
    \begin{align}\label{eq:u_Markov}
        u(t,\bx)
        = \EE\Big[\EE_{t+h}\Big[\phi\circ\frF\Big(\Big\{\bx_s + J^{t, t+h}_s+ J^{ t+h,T}_s:s\in\Dd\Big\}\Big)\Big]\Big]
        = \EE\Big[ u\big( t+h,\bx + J^{t, t+h}\big)\Big].
    \end{align}
    We wish to apply the functional Itô formula to this process, hence we reframe it for that purpose:
    \begin{align*}
        \bx_s+J^{t, t+h}_s &= \left(\bx_s +\int_t^s K(s,r)\D W_r \right)\one_{s\le t+h} + \left(\bx_s+\int_t^{t+h} K(s,r)\D W_r \right)\one_{s> t+h}\\
        &= \bm{X}_s \otimes_{t+h} \bm{\Theta}^{t+h}_s.
    \end{align*}
    We can thus rely on the Itô formula~\eqref{eq:Ito} on the interval~$[t,t+h]$, with time variable frozen, and use~$J^{t,t}_s=0$ to derive
    \begin{align*}
        u\left( t+h,\bx+J^{t, t+h}\right) &=u( t+h,\bx) + \int_t^{t+h} \Big\langle \partial_{\bx} u\left( t+h,\bx+J^{t,r}\right), K^r\Big\rangle \D W_r\\
        &\quad+ \half \int_t^{t+h} \Big\langle \partial^2_\bx u\big( t+h,\bx+J^{t,r}\big), (K^r,K^r)\Big\rangle \dr.
    \end{align*}
    Proposition~\ref{prop:SingularDirections} (which does not rely on Theorem \ref{thm:main}) shows that~$\EE\abs{\langle \partial_{\bx} u\left( t+h,\bx+J^{t,r}\right), K^r\rangle}^2<\infty$ hence the stochastic integral has zero expectation. 
    Equation~\eqref{eq:u_Markov} now paves the way for computing the temporal derivative as follows
    \begin{align}
        \partial_t u(t,\bx)
        &=\lim_{h\to0}\frac{u(t+h,\bx)-u(t,\bx)}{h} \nonumber\\
        &=- \lim_{h\to0}\frac{1}{2h}\EE\int_t^{t+h} \Big\langle \partial^2_\bx u\big( t+h,\bx+J^{t,r}\big), (K^r,K^r)\Big\rangle \dr \label{eq:limit_dtu}\\
        &=-\half \Big\langle \partial^2_\bx u( t,\bx), (K^t,K^t)\Big\rangle. \nonumber
    \end{align}
    The limit is justified thanks to the regularity of~$u$ and by decomposing the pre-limit term as follows:
    \begin{align*}
        &\EE\Big\langle \partial^2_\bx u\big( t+h,\bx+J^{t,r}\big), (K^r,K^r)\Big\rangle - \Big\langle \partial^2_\bx u\big( t,\bx\big), (K^{t},K^{t})\Big\rangle \\
        &= \EE\lim_{\delta\to0} \bigg(\Big\langle D^2_\bx u\big( t+h,\bx+J^{t,r}\big), (K^{\delta,r},K^{\delta,r})\Big\rangle - \Big\langle D^2_\bx u\big( t,\bx\big), (K^{\delta,t},K^{\delta,t})\Big\rangle\bigg)   \\ 
        &=\EE\lim_{\delta\to0} \bigg(\Big\langle D^2_\bx u\big( t+h,\bx+J^{t,r}\big), (K^{\delta,r},K^{\delta,r})\Big\rangle 
        - \Big\langle D^2_\bx u\big( t+h,\bx\big), (K^{\delta,r},K^{\delta,r})\Big\rangle \\
        &\quad+ \Big\langle D^2_\bx u\big( t+h,\bx\big), (K^{\delta,r},K^{\delta,r})\Big\rangle -\Big\langle D^2_\bx u\big( t+h,\bx\big), (K^{\delta,t},K^{\delta,t})\Big\rangle \\
        &\quad + \Big\langle D^2_\bx u\big( t+h,\bx\big), 
        (K^{\delta,t},K^{\delta,t})\Big\rangle
        - \Big\langle D^2_\bx u\big( t,\bx\big), (K^{\delta,t},K^{\delta,t})\Big\rangle\bigg) 
    \end{align*}
    From now on, the constant~$C>0$ may change from line to line. 
    The use of the estimates from Definition~\ref{def:Cplusalpha} are conditional on the directions being supported on a small interval. Viens and Zhang in the proof of \cite[Theorem 3.17]{viens2019martingale} perform a decomposition of~$\one_{[t,T]}$ into smaller intervals, leading to the convergence of~$\langle D_\bx u(t,\bx),K^{\delta,t}\rangle$ thanks to the estimates satisfied by~$u\in\Cc^{0,2}_{\half}(\Lambda)$. This technique however looks insufficient to obtain the convergence of the second derivative in the present setting.
   
    Instead, we exploit the $L^2$--estimates obtained in Lemma~\ref{lemma:DerTwo}. Along with the exponential integrability of~$J^{t,r}$ and the bound~$\abs{K(s,t)}\le C(s-t)^{H-\half}$ from Assumption~\ref{assu:kernel}, they entail
    \begin{align*}
        &\abs{\EE\Big\langle \partial^2_\bx u\big( t+h,\bx+J^{t,r}\big), (K^r,K^r)\Big\rangle - \Big\langle D^2_\bx u\big( t,\bx\big), (K^{t},K^{t})\Big\rangle }\\
        &\le C \lim_{\delta\to0} \bigg(\EE\norm{J^{t,r}}_{L^2[r,\Tb]} \norm{K^{\delta,r}}_{L^2[r,\Tb]}^2
        +\EE\norm{J^{t,r}(K^{\delta,r})^2}_{L^1[r,\Tb]}\\
        &\qquad\qquad
        + \norm{K^{\delta,r}-K^{\delta,t}}_{L^2[t,\Tb]}\norm{K^{\delta,r}+K^{\delta,t}}_{L^2[t,\Tb]}
        \\    &\qquad\qquad+h^H\norm{K^{\delta,r}}_{L^2[t,\Tb]}^2+ \norm{K^{\delta,r}}_{L^2[t,\Tb]}\norm{K^{\delta,r}}_{L^2[t,t+h]}\bigg)\\
        &\le Ch^H,
    \end{align*}
    since $\EE\abs{J^{t,r}_s}^2 \le C(r-t)^{2H}$.   
    The integrand in \eqref{eq:limit_dtu} is thus continuous and the limit is well-defined.
     Finally, even though the underlying process~$J^{t,T}$ is defined over~$\Dd\subset[0,\Tb]$, we do have $J^{T,T}\equiv0$ and thus~$u(T,\bx)=\phi\circ\frF(\bx)$, 
     thereby concluding the first part of the proof.

    Uniqueness of the solution follows from a Feynman-Kac type of argument. Assuming that~$u\in \Cc^{1,2}_\half$ satsifies the PPDE~\eqref{eq:main_PPDE} and the terminal condition, we can apply Itô's formula~\eqref{eq:Ito} and derive
    \begin{align*}
        u(T,\bx+J^{t,T}) 
        &= u(t,\bx) + \int_t^T \partial_t u(s,\bx+J^{s,T})\ds + \int_t^T \Big\langle \partial_\bx u(s,\bx+J^{s,T}),K^s \Big\rangle \D W_s\\
        &\quad+ \half\int_t^T \Big\langle \partial_\bx^2 u(s,\bx+J^{s,T}),(K^s,K^s) \Big\rangle \ds.
    \end{align*}
    Taking expectations and canceling terms thanks to the PPDE, one obtains~$u(t,\bx)=\EE[u(T,\bx+J^{t,T})]=\EE[\phi\circ\frF(\bx+J^{t,T})]$.
\end{proof}
A straightforward application of Theorem \ref{thm:main} is the martingale representation derived from the functional Itô formula.
\begin{corollary}\label{coro:MG_rpz}
    Let Assumptions \ref{assu:coefs} and \ref{assu:kernel} hold. Then for all $0\le t<s\le T$,
    \begin{align}
         u\big(s,\Th^s\big) = u\big(t,\Th^{t} \big) + \int_{t}^s \big\langle \partial_\bx u(r,\Th^r), K^r \big\rangle \D W_r.
    \end{align}
\end{corollary}
\begin{proof}
    This follows from the functional Itô formula~\eqref{eq:Ito} and the PPDE \eqref{eq:main_PPDE} which cancels the finite variation terms.
\end{proof}
The uniqueness of the martingale representation entails that it is equivalent to the Clark-Ocone formula~\cite[Proposition 1.3.14]{nualart2006malliavin}, in particular we have $ \big\langle \partial_\bx u(s,\Th^s), K^s \big\rangle= {\rm D}_s u(s,\Th^s)$ where ${\rm D}$ denotes the Malliavin derivative. From a financial viewpoint, the martingale representation naturally lends itself to hedging formulae, see \cite[Proposition 2.2]{fukasawa2021hedging}. It also shows that a market where the two following assets are traded is complete: the variance derivative with payoff $\phi(V_T)$ and an asset $S$ with dynamics~$\D S_t/S_t = \sigma_s \D B_s$, with $B$ a Brownian motion correlated to $W$. In practical terms, they would correspond to the S\&P 500 and an option on the VIX or on the realised variance.

For completeness, we provide a more explicit expression of the pathwise derivatives at play in this section. The proof can be found in Section \ref{sec:SingularDirections}. For any $h\in C^2(\RR^d,\RR)$, let $\nabla h$ and $\nabla^2 h$ be the gradient and Hessian matrix respectively. For any $y,z\in\RR^{d\times m}$,  we also define for clarity and coherence
\begin{align*}
    \big\langle \nabla h(x), y \big\rangle 
    := \nabla h(x)^\top y, \qquad \text{and} \qquad 
    \big\langle \nabla^2 h(x), (y,z) \big\rangle 
    := {\rm Tr}\left(z^\top \nabla^2 h(x) y\right). 
\end{align*}

\begin{proposition}\label{prop:SingularDirections}
Let Assumptions \ref{assu:coefs} and \ref{assu:kernel} hold.
Then for all $(t,\bx)\in\Lambda$, we have
\begin{align}
    \big\langle \partial_\bx u(t,\bx) ,K^t \big\rangle 
    &=\EE\left[\phi'(V^{t,\bx}_T) \int_{t}^{\Tb} \Big\langle \nabla f_s(\Wh_s^{t,\bx}), K(s,t) \Big\rangle \ds\right], \label{eq:DerOne_Singular}\\
    \big\langle \partial_\bx^2 u(t,\bx) ,(K^t,K^t)\big\rangle
    &=\EE\left[\phi''(V^{t,\bx}_T) \left(\int_{t}^{\Tb}  \Big\langle \nabla f_s(\Wh_s^{t,\bx}), K(s,t) \Big\rangle\ds \right)^2\right]  \label{eq:DerTwo_Singular} \\
    &\quad +\EE\left[ \phi'(V^{t,\bx}_T) \int_{t}^{\Tb} \Big\langle \nabla^2 f_s(\Wh_s^{t,\bx}), \big( K(s,t),K(s,t)\big)\Big\rangle\ds 
    \right].\nonumber
\end{align}    
\end{proposition}
\begin{remark}
   Again, a parallel can be drawn with Malliavin calculus as we have
    \begin{align*}
         \big\langle \partial_\bx u(t,\bx) ,K^t \big\rangle 
    =\EE\left[{\rm D}_t\phi(V^{t,\bx}_T)\right],
    \qquad \text{and}\qquad 
    \big\langle \partial_\bx^2 u(t,\bx) ,(K^t,K^t) \big\rangle 
    =\EE\left[{\rm D}_t^2 \phi(V^{t,\bx}_T)\right],
    \end{align*}
    where ${\rm D}$ and ${\rm D}^2$ are respectively the first and second Malliavin derivatives.
\end{remark}

\section{The implied volatility counterpart}\label{sec:IV}
Inspired by \cite{berestycki2004computing}, we derive a path-dependent PDE for the implied volatility of the volatility derivatives we considered so far. In this paper, Berestycki, Busca and Florent considered the implied volatility in a (Markovian) stochastic volatility model, which yields a pricing PDE on~$\RR_+\times\RR^n$. Even in this simpler setting, well-posedness was a highly challenging task therefore we do not intend to pursue such a goal here, since the solution theory of our PPDEs remains limited. 

Let $\BS(t,S,\sigma;T,\kappa)$ be the Black-Scholes price at time~$t\ge0$ of a Call option with spot~$S>0$, volatility~$\sigma>0$, maturity~$T>0$, strike~$\kappa>0$ and with zero interest rates.
The Black-Scholes formula gives
$$
\BS(t,S,\sigma;T,\kappa) = S N\left(\frac{\log(S/\kappa)}{\sigma\sqrt{T-t}}+\half\sigma\sqrt{T-t}\right)- \kappa N\left(\frac{\log(S/\kappa)}{\sigma\sqrt{T-t}}-\half\sigma\sqrt{T-t}\right).
$$
Let~$\varphi(t,\bx)$ be the price of the traded future at time~$t\in[0,T]$: the VIX future is~$\varphi(t,\bx)=\EE\Big[\sqrt{V^{t,\bx}_T/\Delta}\Big]$, for~$\Delta=30$ days, while the variance swap gives~$\varphi(t,\bx)=\EE[V^{t,\bx}_T/T]$ (for different~$V_T^{t,\bx}$, see Example~\ref{ex:derivatives}). Note that if $V$ were a traded asset (e.g. a stock price) then it would be a martingale and~$\EE[V_T^{t,\bx}]=V_t^{t,\bx}=\bx_t$. We assume, as it is the case for both the VIX future and the variance swap, that the payoff belongs to~$C^3_{{\rm poly}}(\RR)$ such that $\varphi$ is a particular case of Theorem~\ref{thm:main} and solves the path-dependent PDE
\begin{align*}
\partial_t \varphi(t,\bx) + \half\langle \partial^2_\bx \varphi(t,\bx),(K^t,K^t)\rangle =0 ,\qquad \text{for all   }(t,\bx)\in\Lambda.
\end{align*}
Let us define the reduced variable~$\psi(t,\bx) := \log\Big(\varphi(t,\bx)/\kappa\Big)$,
which hence solves for all~$(t,\bx)\in\Lambda$
\begin{align}\label{eq:PDE_reduced_future}
    \partial_t \psi(t,\bx) + \half\Big( \langle \partial^2_\bx \psi(t,\bx),(K^t,K^t)\rangle +\abs{ \langle \partial_\bx \psi(t,\bx),K^t\rangle}^2 \Big) =0 ,
\end{align}
where~$\abs{\cdot}$ denotes the Euclidean norm in~$\RR^m$.
We notice that the Call price can be written as
\begin{equation}\label{eq:CallPrice}
u(t,\Th^t) = \EE\big[\big(\varphi(T,\Th^T)-\kappa\big)_+ |\Ff_t\big] \quad \text{and}\quad u(t,\bx) = \EE\big[\big(\varphi(T,\bx+J^{t,T})-\kappa\big)_+\big],
\end{equation}
since~$\Th^T-\Th^t=J^{t,T}_s$ is independent of~$\Ff_t$, see Equation~\eqref{eq:def_JtT} and below.
We do not know at this point how to exploit the regularising property of the expectation to relax the assumption~$\phi\in C^3$.
Because of the non-differentiability of the function~$x\mapsto (x)_+$, this map~$u$ is not covered by Theorem~\ref{thm:main}, hence we will make this a standing assumption, from now on and in Theorem~\ref{th:IV}. The future price~$\phi(t,\Th^t)$ is a true martingale, by integrability of~$V_T$, thus the implied volatility of the variance derivative~$V_T$ is defined as the unique non-negative solution~$\Sigma=\Sigma(t,\bx;T,\kappa)$ to
\begin{equation*}\label{eq:def_IV1}
    u(t,\bx) =\BS(t,\varphi(t,\bx),\Sigma;T,\kappa).
\end{equation*}
The main difference with~\cite{berestycki2004computing} is that the future price also depends on $(t,\bx)$ which brings more intricate dependencies. Although some asymptotic results for implied volatility of VIX exist, the main result of this section is to the best of our knowledge the first theoretical result that holds at any time~$t\ge0$ and without involving~$\BS^{-1}$.
\begin{theorem}\label{th:IV}
We assume that the functional~$u:\Lambda\to\RR$ defined in~\eqref{eq:CallPrice} belongs to~$\Cc^{1,2}_{\alpha}(\Lambda)$ for some~$\alpha\in(0,1)$.
For all~$(t,\bx)\in\Lambda$, the total implied variance~$\wS\equiv\wS(t,\bx):= (T-t) \Sigma(t,\bx)^2$ is a~$\Cc^{1,2}(\Lambda)$ solution of the PPDE
\begin{align}\label{eq:IV_PPDE}
    \partial_t\wS + \frac12 \left(\big(\partial_\bx\wS\big)^\top \partial_\bx\psi 
     + \partial^2_\bx \wS  \right) -\left(\frac{1}{16}+\frac{1}{4\wS}\right)
    \big\lvert\partial_\bx\wS\big\lvert^2 + \Big\lvert \partial_{\bx}\psi - \frac{\psi}{2\wS}\partial_\bx\wS\Big\lvert^2 =0,
\end{align}
where~$\abs{\cdot}$ represents the Euclidean norm in~$\RR^m$ and for all~$\mathfrak{u}:\Lambda\to\RR$ we denote~$\partial_\bx \mathfrak{u} = \langle \partial_\bx \mathfrak{u}(t,\bx), K^t \rangle$ and~$\partial_\bx^2 \mathfrak{u} = \langle \partial_\bx^2 \mathfrak{u}(t,\bx), (K^t,K^t) \rangle$.
\end{theorem}
\begin{remark}
Letting~$\tau=T-t$, this is equivalent to the PPDE for the implied volatility
\begin{align*}
    \partial_\tau (\tau\Sigma^2)
    &= \tau\Sigma \,\big(\partial_\bx\Sigma\big)^\top \partial_\bx\psi + 2 \tau\big( \Sigma \partial_\bx^2 \Sigma+\abs{\partial_\bx \Sigma}^2 \big) 
    -  \left(\frac{\tau^2\Sigma^2}{4}+\tau \right) \abs{\partial_\bx \Sigma}^2 
    +\abs{ \partial_{\bx}\psi - \frac{\psi}{\Sigma} \partial_\bx\Sigma}^2\\
    \text{or} \quad \partial_\tau (\tau\Sigma^2)&= \big(\partial_\bx \Sigma\big)^\top\partial_\bx\psi \left(\tau \Sigma  -2\frac{\psi}{\Sigma}\right) + \abs{\partial_\bx \Sigma}^2\left( \frac{\tau^2\Sigma^2}{4} + \tau +\frac{\psi^2}{\Sigma^2} \right) + 2\tau\Sigma \,\partial^2_\bx \Sigma + \abs{\partial_\bx \psi}^2.
\end{align*}
Setting $t=T$, i.e. $\tau=0$, we are left with
$\displaystyle
\Sigma^2(T,\bx) = \Big\lvert\partial_{\bx}\psi - \frac{\psi}{\Sigma}\partial_\bx\Sigma\Big\lvert^2 (T,\bx).
$
\notthis{
An attempt at simplifying this equation leads us to define $\Psi:=\psi/\Sigma$ which solves
$$
\Psi^2(T,\bx) = \psi^2\abs{\partial_\bx \psi - \Psi \frac{\psi\partial_\bx \Psi + \Psi \partial_\bx \psi}{\psi^2} }^2 (T,\bx).
$$    
}
Clearly both PPDEs, for~$t\in[0,T)$ and $t=T$,  are well-posed in the sense that~$\Sigma$ is a classical solution. As mentioned above, whether this system of PDEs has a unique solution was already a hard problem in the Markovian case and is out of scope of this article.
\end{remark}
\begin{proof}
The proof follows~\cite[Section 3.1]{berestycki2004computing}, only with more details.
Let us introduce the reduced variables~$\tau=T-t$ and $x=\log(S/\kappa)$ and define the unit volatility Black-Scholes price
$$
v(\tau,x) = \E^x N\left(\frac{x}{\sqrt{\tau}}+\half\sqrt{\tau}\right)-N\left(\frac{x}{\sqrt{\tau}}-\half\sqrt{\tau}\right).
$$
We observe that
$
\BS(t,S,\sigma;T,\kappa)= \kappa v (\sigma^2(T-t), \log(S/\kappa))
$
and $v$ solves the initial value problem
\begin{equation}\label{eq:PDE_reduced}
\def\arraystretch{1.5}
    \left\{\hspace{-.4cm}
\begin{array}{rl}
    &\partial_\tau v(\tau,x) = \half (\partial_x^2-\partial_x) v(\tau,x), \qquad \tau>0, \,x \in \RR, \\
    &v(0,x) = (\E^x-1)_+.
\end{array}
\right.
\end{equation}
The implied volatility can be defined via this reduced approach to
\begin{equation}\label{eq:def_IV}
    u(t,\bx) =\BS(t,\varphi(t,\bx),\Sigma;T,\kappa)=\kappa v\big(\Sigma^2(T-t),\psi(t,\bx) \big).
\end{equation}
We write for conciseness $u=u(t,\bx),\, v=v(\Sigma^2(T-t),\psi(t,\bx)),\,\psi=\psi(t,\bx)$ and the total implied variance~$\widehat{\Sigma}=\Sigma(t,\bx;T,\kappa)^2(T-t)$. 
\notthis{
The implied volatility is differentiable using the chain rule since~$\partial_\bx u$ exists and~$\partial_\sigma \BS>0$ . It is given by
$$
\partial_\bx \Sigma = \frac{\partial_\bx \BS(t,\varphi,\Sigma) - \partial_x \BS(t,\varphi,\Sigma) \partial_\bx \varphi}{\partial_\sigma \BS(t,\varphi,\Sigma)}= \frac{\partial_\bx u - \partial_x \BS(t,\varphi,\Sigma) \partial_\bx \varphi}{\partial_\sigma \BS(t,\varphi,\Sigma)},
$$
where all the derivatives on the right-hand-side are well defined. The same conclusion also holds for~$\partial_t$ and~$\partial^2_\bx$.}
Through the informal chain rule formula applied to~$u$, we can define the derivatives~$\partial \wS$ for~$\partial\in\{\partial_t, \partial_\bx,\partial^2_\bx\}$. Indeed, all the derivatives of~$u,v$ and~$\varphi$ are already well-defined and~$\partial_\tau v>0$, hence the derivatives of~$\wS$ are uniquely determined by:
\begin{align*}
    \frac{\partial_t u}{\kappa}& = \partial_\tau v \, \partial_t\wS + \partial_x v \,\partial_t\psi;\qquad 
    \frac{\partial_\bx u}{\kappa} = \partial_\tau v \, \partial_\bx\wS + \partial_x v \, \partial_\bx \psi; \\
    \frac{\partial^2_\bx u}{\kappa} &= 2\partial_{\tau x} v\, (\partial_\bx\wS)^\top \,\partial_\bx \psi + \partial_\tau v \,\partial^2_\bx \wS + \partial^2_{\tau}v \,\big\lvert\partial_\bx\wS\big\lvert^2 + \partial^2_x v \,\abs{\partial_\bx \psi}^2 + \partial_x v \,\partial^2_\bx \psi,
\end{align*}
where we recall that~$\partial_\bx u,\partial_\bx \psi$ and~$\partial_\bx \wS$ are~$\RR^m$-valued.
It turns out these are also the derivatives of interest from~\eqref{eq:def_IV}, yielding the expression
\begin{align}\label{eq:Computation_v}
    \left(\partial_t + \half \partial_\bx^2 \right) u &= \kappa \partial_\tau v \bigg\{ 
    \partial_t\wS 
    +\frac{\partial_x v}{\partial_\tau v}   \left( \partial_t\psi + \half \partial^2_\bx \psi\right)
    + \frac{\partial_{\tau x} v}{\partial_\tau v}  \big(\partial_\bx\wS\big)^\top \partial_\bx\psi \\
    &\qquad \qquad 
    + \half\partial^2_\bx \wS  
    + \half \frac{\partial^2_{\tau}v}{\partial_\tau v} \abs{\partial_\bx\wS}^2
    +\half \frac{ \partial^2_x v}{\partial_\tau v}\abs{\partial_\bx\psi}^2
    \bigg\}. \nonumber
\end{align}
We use the relation~$\frac{ \partial^2_x v}{\partial_\tau v} =  2+\frac{ \partial_x v}{\partial_\tau v}$ from~\eqref{eq:PDE_reduced} and the PPDE~\eqref{eq:PDE_reduced_future} to get
$$
\frac{\partial_x v}{\partial_\tau v}   \left( \partial_t\psi + \half \partial^2_\bx \psi\right) + \half \frac{ \partial^2_x v}{\partial_\tau v}\abs{\partial_{\bx}\psi}^2
= \frac{\partial_x v}{\partial_\tau v} \left(\partial_t\psi + \half \partial^2_\bx \psi + \half\abs{\partial_{\bx}\psi}^2\right)+\abs{\partial_{\bx}\psi}^2=\abs{\partial_{\bx}\psi}^2.
$$
Furthermore, we plug the classical relations of the Black-Scholes greeks
\begin{align*}
    \frac{\partial_{\tau x} v}{\partial_\tau v}(\tau,x) =  \half-\frac{x}{\tau}\qquad\text{and}\qquad 
    \frac{\partial_{\tau}^2 v}{\partial_\tau v}(\tau,x) = -\frac{1}{2 \tau}  + \frac{x^2}{2\tau^2} - \frac18
\end{align*}
in~\eqref{eq:Computation_v} to obtain
\begin{align*}
    \left(\partial_t + \half \partial_\bx^2 \right) u = \kappa \partial_\tau v \bigg\{ &\partial_t\wS +\abs{\partial_{\bx}\psi}^2+ \left(\frac12-\frac{\psi}{\wS}\right)\big(\partial_\bx\wS\big)^\top \,\partial_\bx\psi 
     + \half\partial^2_\bx \wS  \\
    &+ \half \left(-\frac{1}{2\wS}+\frac{\psi^2}{2\wS^2}-\frac18\right)
    \abs{\partial_\bx\wS}^2 
    \bigg\}.
\end{align*}
This simplifies further when observing that
$$
\abs{\partial_{\bx}\psi}^2- \frac{\psi}{\wS} \big(\partial_\bx\wS\big)^\top \,\partial_\bx\psi +\frac{\psi^2}{4\wS^2} \abs{\partial_\bx\wS}^2 = \abs{\partial_{\bx}\psi - \frac{\psi}{2\wS}\partial_\bx\wS}^2.
$$
Since the left-hand-side~$\left(\partial_t + \half \partial_\bx^2 \right) u$ is equal to zero we reach the PPDE for the total implied variance
\begin{align*}
    \partial_t\wS + \frac12 \left(\big(\partial_\bx\wS\big)^\top \partial_\bx\psi 
     + \partial^2_\bx \wS  \right) -\left(\frac{1}{16}+\frac{1}{4\wS}\right)
    \abs{\partial_\bx\wS}^2 + \abs{ \partial_{\bx}\psi - \frac{\psi}{2\wS}\partial_\bx\wS}^2 =0,
\end{align*}
which is precisely the claim.
\end{proof}


\section{Path but no past in the Markovian case}
\label{sec:Markov}
This subsection explores the particular case of VIX options, technically speaking the case $\Dd=[T,\Tb]$, which simplifies considerably in Markovian models. The latter corresponds to the first bullet point in Example~\ref{ex:kernels} with the exponential kernel~$K(t,r)=\E^{-\beta(t-r)} \nu$ where~$\nu\in\RR^{d\times m}$, $\beta\in \RR_+^{d\times d}$ and the matrix exponential is understood componentwise, that is~$\E^\beta:=(\E^{\beta_{ij}})_{i,j=1}^d$. Hence, setting~$\gamma_t=\E^{-\beta t}\gamma_0$ for all~$t\in[0,\Tb]$ with~$\gamma_0\in\RR^d$, $\Wh$ is an Ornstein-Uhlenbeck process solution to the SDE~$\D X_t = -\beta X_t \dt + \nu \,\D W_t$ and~$X_0=\gamma_0\in\RR^d$. Under the assumption that the payoff does not look into~$[0,T)$, the option price at time~$t\in[0,T]$ is only a function of~$t$ and~$\Wh_t$ (Proposition~\ref{prop:Markov_rpz}), therefore no path-dependent lift is necessary and this price is recovered as the solution of a finite-dimensional PDE (Corollary~\ref{coro:Markov_PDE}). Although similar in spirit to~\cite{barletta2019short,fouque2018heston,lin2009vix,papanicolaou2014regime} for the Heston model, the present case involves an additional non-linearity~$f$ which prevents from computing~$V_T$ explicitly.
After the theoretical results, numerical illustrations complete the picture.
\begin{proposition}\label{prop:Markov_rpz}
    Let~$\phi:\RR\to\RR$ and~$f:\Dd\times\RR\to\RR$ be measurable functions and~$\Dd=[T,\Tb]$. There exists a measurable map~$\tilde{u}:[0,T]\times\RR^d \to\RR$ such that
    $$
\EE_t[\phi(V_T)] = \tilde{u}(t, \Wh_t), \qquad \text{for all   } t\in[0,T].
    $$
\end{proposition}
\begin{proof}
    We simply have to notice that for all~$0\le t\le T\le s \le \Tb$, $\Th^t_s = \E^{-\beta s}\gamma_0+\int_0^t \E^{-\beta(s-r)} \nu \D W_r = \E^{-\beta(s-t)} \Wh_t$. Then, Equation~\eqref{eq:uMarkov} reads
    $$
\EE_t\big[\phi(V_t)\big] = \EE_t\Big[\phi\circ\frF\Big(\Big\{ \E^{-\beta(s-t)} \Wh_t + J^{t,T}_s: s\in\Dd \Big\}\Big)\Big],
    $$
    where~$J^{t,T}_s$ is independent of~$\Ff_T$. Therefore there exists a map~$\tilde{u}:[0,T]\times\RR^d\to\RR$ such that~$\EE_t[\phi(V_T)] = \tilde{u}(t,\Wh_t)$ almost surely. This map is defined for all~$t\in[0,T]$ and $w\in\RR^d$ as
    \begin{equation}\label{eq:Def_u_Markov}
        \tilde{u}(t,w) := \EE\Big[\phi\circ\frF\Big(\big\{ \E^{-\beta(s-t)} w + J^{t,T}_s: s\in\Dd \big\}\Big)\Big],
    \end{equation}
    which concludes the proof.
\end{proof}
Classical results of Feynman-Kac type asserts that this pricing function satisfies a finite-dimensional backward Kolmogorov PDE starring the generator of the OU process. However, we can also retrieve it from the PPDE, thereby exploiting the Markov property of the OU process. First compare~\eqref{eq:Def_u} and~\eqref{eq:Def_u_Markov} and notice that, for all~$(t,w)\in[0,T]\times\RR^d$,
    \begin{equation}\label{eq:Eq_u_baru}
    \tilde{u}(t,w) = \EE\Big[\phi\circ\frF\Big(\big\{ \E^{-\beta(s-t)} w + J^{t,T}_s: s\in\Dd \big\}\Big)\Big]
    = u(t, \E^{-\beta(\cdot-t)}w).
    \end{equation}
    From the terminal condition of the PPDE~\eqref{eq:main_PPDE} we immediately recover the terminal condition of the PDE~\eqref{eq:Markov_PDE}.
    Let~$(\bt,\bw)\in[0,T]\times\RR^d$ and~$z\in\RR^{d\times m}$. Notice that~$K^{\bt} = \E^{-\beta(\cdot-\bt)} \nu$ is continuous. Hence there is no need for a truncation argument, $\langle D_{\bx} u(t,\bx),K^t\rangle$ is well-defined as a Fréchet derivative and is equal to~$\langle \partial_{\bx} u(t,\bx),K^t\rangle$. Let us look carefully at the derivatives: 
    \begin{align}\label{eq:Relation_Derivatives_Markov}
    \big\langle \nabla \tilde{u}\big(\bt,\bw \big),  z\big\rangle 
    &= \big\langle \partial_{\bw} u\big(\bt,\E^{-\beta(\cdot-\bt)} \bw \big),  z\big\rangle 
    = \big\langle \partial_\bx u\big(\bt,\E^{-\beta(\cdot-\bt)} \bw \big),   \E^{-\beta(\cdot-\bt)}z\big\rangle\\
    \big\langle \nabla^2 \tilde{u}(\bt, \bw ), (\nu,\nu) \big\rangle
    &=\big\langle\partial^2_\bx u\big(\bt,\E^{-\beta(\cdot-\bt)} \bw \big), (K^{\bt},K^{\bt}) \big\rangle, \nonumber\\
    \partial_{\bt} \tilde{u}(\bt,\bw) 
    &= \partial_{\bt} u\big(\bt,\E^{-\beta(\cdot-\bt)} \bw \big)
    = \partial_t u\big(\bt,\E^{-\beta(\cdot-\bt)} \bw \big) + \big\langle \partial_\bx u\big(\bt,\E^{-\beta(\cdot-\bt)} \bw \big), \beta\E^{-\beta(\cdot-\bt)} \bw \big\rangle \nonumber \\
    & =  \partial_t u\big(\bt,\E^{-\beta(\cdot-\bt)} \bw \big) +\big\langle \nabla \tilde{u}\big(\bt,\bw \big),  \beta\bw\big\rangle \nonumber
    \end{align}
    The PPDE~\eqref{eq:main_PPDE} holds for any~$\bx\in C^0 ([0,\Tb], \mathbb{R}^d)$, in particular~$\bx = \E^{\beta(\cdot-\bt)} \bw$ for any~$\bw\in\RR^d$. Therefore we obtain
    \begin{align*}
        &\partial_{\bt} \tilde{u}(\bt,\bw) - \big\langle \nabla \tilde{u}\big(\bt,\bw \big),  \beta\bw\big\rangle + \half\big\langle \nabla^2 \tilde{u}(\bt, \bw ), (\nu,\nu) \big\rangle\\
        &\qquad=\partial_t u\big(\bt,\E^{-\beta(\cdot-\bt)} \bw \big) + \half \big\langle\partial^2_\bx u\big(\bt,\E^{-\beta(\cdot-\bt)} \bw \big), (K^{\bt},K^{\bt}) \big\rangle = 0.
    \end{align*}    
    Combined with the terminal condition, this entails both existence and uniqueness of the PDE.
\begin{corollary}\label{coro:Markov_PDE}
    Let Assumption~\ref{assu:coefs} hold. The value function~$\tilde{u}$ defined in~\eqref{eq:Def_u_Markov} is the unique~$C^{1,2}_{{\rm exp}+}([0,T]\times\RR^d)$ solution to the partial differential equation, for all~$(t,w)\in[0,T]\times\RR^d$,
    \begin{equation}\label{eq:Markov_PDE}
\def\arraystretch{1.5}
    \left\{\hspace{-.4cm}
\begin{array}{rl}
    &\displaystyle \partial_t \tilde{u}(t,w)- \big\langle \nabla \tilde{u}\big(t,w \big),  \beta w\big\rangle  + \half \big\langle \nabla^2 \tilde{u}(t,w), (\nu,\nu) \big\rangle = 0,\\
    &\displaystyle  \tilde{u}(T,w) = \phi\circ\frF\big(\big\{ \E^{-\beta(s-T)} w : s\in\Dd\big\}\big),
\end{array}
\right.
\end{equation}
where~$\frF$ is defined in~\eqref{eq:Def_frF_bis}.
\end{corollary}
\begin{remark}
   Equation~\eqref{eq:Eq_u_baru} combined with Corollary \ref{coro:MG_rpz} and Proposition \ref{prop:SingularDirections} yield formulae for the greeks.
\end{remark}
\begin{remark}
    The results of this section still hold if one replaces~$\Wh$ with a different semimartingale, however the Gaussian case has the advantage that~$\frF$ is given in integral form because~$I^t_s$ is Gaussian. For non-Gaussian processes, this function and hence the terminal condition are in general computable only by Monte Carlo.
\end{remark}

\subsection{Numerical example}
Let us specify a model under which we will illustrate the option prices given by the PDE method.
Let $d=m=2$, for all~$\nu_1,\nu_2,\nu_3\in\RR$, $\beta_1,\beta_2,\beta_3>0$ and~$\rho\in[-1,1],\,\bar{\rho}=\sqrt{1-\rho^2}$, we define the triangular matrices~$\nu\in\RR^{2\times 2}$ and~$\beta\in\RR^{2\times 2}_+$:
\begin{align*}
\nu = \begin{bmatrix}
       \bar{\rho} \nu_1 & \rho\nu_2 \\ 0 & \nu_3
    \end{bmatrix}  \quad \text{and} \quad
    \beta =  \begin{bmatrix} \beta_1 & \beta_2 \\
    0 & \beta_3 \end{bmatrix},\quad \text{such that} \quad
    \E^\beta = \begin{bmatrix} \E^{\beta_1} & \E^{\beta_2} \\ 0 & \E^{\beta_3}
    \end{bmatrix}.    
\end{align*}
This yields
\begin{align*}
   K(t) =  \begin{bmatrix}
         \bar{\rho}\nu_1 \E^{- \beta_1 t} &\rho\nu_2\E^{-\beta_2 t} \\
        0 & \nu_3\E^{-\beta_3 t}
    \end{bmatrix}
    \quad \text{and thus}\quad
    \Wh_t = \begin{bmatrix}
        \int_0^t \nu_1 \E^{-\beta_1 (t-r)}\bar{\rho} \D W^1_r + \int_0^t \nu_2 \E^{-\beta_2(t-r)} \rho \D W^2_r  \\
        \int_0^t \nu_3 \E^{-\beta_3 (t-r)} \D W^2_r
    \end{bmatrix}.
\end{align*}
In the original Bergomi models~\cite{bergomi2008smile2,bergomi2008smile3}, the volatility is a function of~$\Wh^1$ albeit with~$\beta_1\neq\beta_2$, while more modern versions such as in~\cite{romer2022empirical} consider a weighted sum of nonlinear functionals of~$\Wh^1$ and~$\Wh^2$, with~$\beta_1=\beta_2$ and~$\nu_1=\nu_2$. In this section we will consider the latter. Then following Example~\ref{ex:volfct} we define 
$$
f_t(w) = \zeta(t) \left(\lambda \exp\left(w_1-\frac{\nu_1^2(1-\E^{-2\beta_1 t})}{4\beta_1}\right) + (1-\lambda) \exp\left(w_2-\frac{\nu_3^2(1-\E^{-2\beta_3 t})}{4\beta_3}\right) \right), 
$$
where we set~$\zeta\equiv\zeta(0)>0$ for simplicity and~$\lambda\in(0,1)$. 
Note that the correlation~$\rho$ is not visible here, but acts in the PDE through the matrix~$\nu$. If we had picked~$\beta_1\neq\beta_2$ then~$\rho$ would also be present in~$f$. We are concerned with VIX options hence we set~$\Dd=[T,T+\Delta]$ where~$\Delta$ is the 30 days window. 
For pricing a Call with strike~$K>0$ we set~$\phi(x)=(\sqrt{x/\Delta}-K)_+$ and recall that~$\frF(\{\bx_s:s\in\Dd\})=\int_\Dd \EE[f_s(\bx_s+I_s^T)] \ds$. 

For computing the terminal condition, we observe that~$I^T_s=(\EE_T[\Wh^1_s], \, \EE_T[\Wh^2_s])^\top$ and let~$\bx=\E^{-\beta(\cdot-T)} w$ such that we have 
\begin{align*}
&\EE[f_s(\E^{-\beta(s-T)} w + I^T_s)] \\
&= \zeta \bigg\{\lambda \exp\left(\E^{-\beta_1(s-T)}w_1 -\frac{\nu_1^2(1-\E^{-2\beta_1 s})}{4\beta_1}\right)\EE\E^{\EE_T[\Wh^1_s]} 
+ (1-\lambda) \exp\left(\E^{-\beta_3(s-T)}w_2 -\frac{\nu_3^2(1-\E^{-2\beta_3 s})}{4\beta_3}\right)\EE\E^{\EE_T[\Wh^2_s]} \bigg\} \\
&= \zeta \bigg\{\lambda \exp\left(\E^{-\beta_1(s-T)}w_1 +\frac{\nu_1^2\E^{-2\beta_1 s}}{4\beta_1}(1-\E^{-2\beta_1 T})\right)
+ (1-\lambda) \exp\left(\E^{-\beta_3(s-T)}w_2 +\frac{\nu_3^2\E^{-2\beta_3 s}}{4\beta_3}(1-\E^{-2\beta_3T})\right)\bigg\}.
\end{align*}
This function is then integrated between $T$ and~$T+\Delta$ using Simpson's rule. We propagate this terminal condition through time via an explicit finite-difference scheme on a two-dimensional grid. The choice of boundary conditions requires more thought. By the form of the function~$f_t(w)$ it is apparent that the option price should vanish as~$w_1$ and~$w_2$ both tend to~$-\infty$; it is however unclear how the price shall behave in other limits and with respect to the interaction between~$w_1$ and~$w_2$ when the problem is not symmetric. The pronounced convex shape of the solution (see Figure~\ref{fig:VIX_pde}) seems unfit for Dirichlet or Neumann conditions but invites us to move one order further and impose that the second derivative is zero at the boundary. This latter choice indeed brings the most consistency. Furthermore, for the sets of parameters that we experienced with, the best stability is achieved when choosing the number of time steps~$m$ and the number of points on one slice of the grid~$n$ to be equal. The complexity of the algorithm is thus~$\Oo(n^2m)=\Oo(n^3)$. Better stability and convergence could be achieved with an implicit or Crank-Nicolson method; we leave this more thorough analysis for future research.

In this particular example, the conditional expectation can be expressed analytically therefore Monte Carlo schemes are rather straightforward; one simply needs an integration step to recover the payoff. This method has the advantage of computing at almost no cost option prices with new payoff functions since one can reuse the same simulated trajectories, while the PDE method requires to start over. On the other hand, modifying the initial condition entails simulating new paths while the PDE method outputs a price for every spatial point on the grid. The initial condition~$w$ of the OU process does not have an obvious financial interpretation. Nevertheless, we recall point (5) of Remark~\ref{rem:mainPDE} which details how, in the one-dimensional case, one can recover~$\Theta^t_s$ (and hence~$w$) from the forward variance curve. In addition, the PDE scheme offers prices for all times to maturity on the grid as well as derivatives which can be leveraged to compute greeks.

We set the following parameters: $T=1, \Delta=30/365$, $\rho=0.5, \zeta=1, \lambda=0.2, \beta_1=0.2,\beta_3=0.4,\nu_1=0.3,\nu_3=0.5$, with~$m=100$ time steps and a grid of size~$10$ with~$n=101$ points in each dimension. We compute a Call price with strike~$K=1$. To illustrate our results, we present in Figure~\ref{fig:VIX_pde} slices of the solution to the pricing PDE, one with fixed time and one with fixed~$w_2$.
\begin{figure}[!htbp]
\scalebox{0.8}{\includegraphics[center]{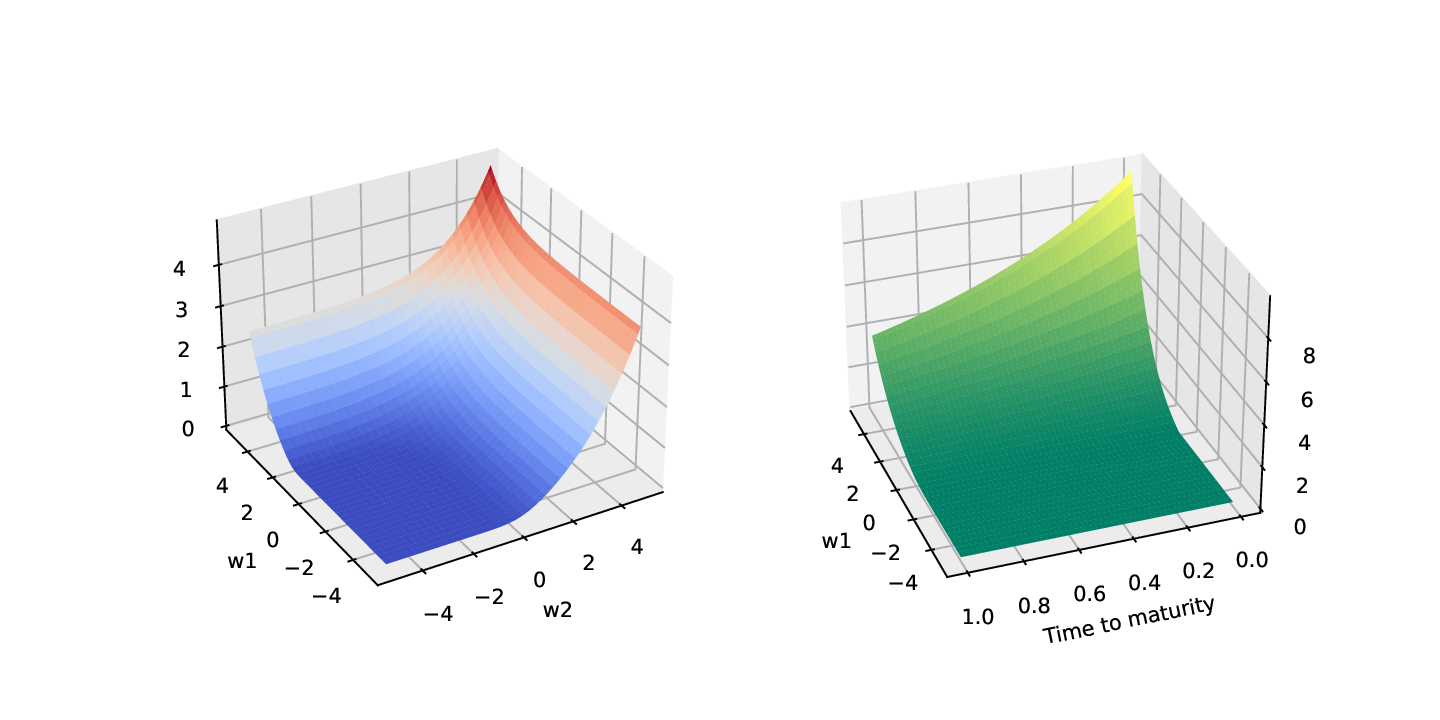}}
\caption{
VIX option prices computed with an explicit finite-difference scheme in a two-factor Bergomi model; $(w_1,w_2)$ is the starting point of the OU process. On the left, the price is expressed with respect to~$w_1,w_2$ at time~$0$. On the right, it is shown with respect to~$w_1$ and time to maturity, with~$w_2=0$ fixed. 
}
\label{fig:VIX_pde}
\end{figure}

\subsection{Implied volatility}
The dimension reduction also applies to the implied volatility PPDE investigated in Section~\ref{sec:IV}. Recalling the definitions of~$\varphi, \psi, \,u$ and~$\Sigma$ from this section, let~$\kappa>0$, $\tilde{\varphi}(t,w) := \varphi(t,\E^{-\beta(\cdot-t)} w)$, $\tilde{\psi}(t,w) := \psi(t,\E^{-\beta(\cdot-t)} w)=\log(\tilde{\varphi}(t,w)/\kappa)$ and~$\tilde{u}(t,w) := u(t,\E^{-\beta(\cdot-t)} w)$. Define~$\widetilde{\Sigma}$ to be the VIX implied volatility in this setting, that is the unique solution to~$\tilde{u}(t,w) = \BS(t,\tilde{\varphi}(t,w), \widetilde{\Sigma};T,\kappa)$. By identification we deduce that~$\Sigma(t,w) = \widetilde{\Sigma}(t,\E^{-\beta(\cdot-t)} w)$ and similarly we define the implied variance~$\widehat{\widetilde{\Sigma}}(t,w):=(T-t) \widetilde{\Sigma}^2(t,w)=\widehat{\Sigma}(t,\E^{-\beta(\cdot-t)} w)$. From the PPDE~\eqref{eq:IV_PPDE} and the relations between the derivatives that we computed in~\eqref{eq:Relation_Derivatives_Markov} we infer that the implied variance solves for all~$(t,w)\in[0,T]\times\RR^d$ the PDE
$$
\partial_t \widehat{\widetilde{\Sigma}} - \big\langle \nabla \widehat{\widetilde{\Sigma}}, \beta w \rangle + \frac12 \left(\big(\nabla \widehat{\widetilde{\Sigma}}\big)^\top \nabla \tilde{\psi} 
     + \nabla^2 \widehat{\widetilde{\Sigma}}  \right) -\left(\frac{1}{16}+\frac{1}{4\widehat{\widetilde{\Sigma}}}\right)
    \big\lvert\nabla\widehat{\widetilde{\Sigma}}\big\lvert^2 + \Big\lvert \nabla  \tilde{\psi} - \frac{ \tilde{\psi}}{2\widehat{\widetilde{\Sigma}}}\nabla\widehat{\widetilde{\Sigma}}\Big\lvert^2 =0,
$$
where~$\abs{\cdot}$ stands for the Euclidean norm in~$\RR^m$ and for both~$\mathfrak{u}\in\{\widehat{\widetilde{\Sigma}}, \tilde{\psi} \}$ we denote  $\nabla \mathfrak{u} = \langle \nabla \mathfrak{u}, \beta \rangle$ and~ $\nabla^2 \mathfrak{u} = \langle \nabla^2 \mathfrak{u}, (\nu,\nu) \rangle$.
\section{Proofs of Section~\ref{sec:PPDE}}\label{sec:proof}

\subsection{Useful estimates}
The derivation, growth and regularity estimates of the partial derivatives of~$u$ hinge on recurrent estimates, summarised in Lemmas \ref{lemma:Estimates_fg} and \ref{lemma:Limit_payoff} below. Recall the definitions of $\Wh^{t,\bx}$ and~$V^{t,\bx}$ from~\eqref{eq:defs_tx}. 
\begin{lemma}\label{lemma:Estimates_fg}
    Let~$g\in C^0_{{\rm poly}}((0,+\infty))$, $h\in C^{0,0}_{{\rm exp}+}([0,\Tb]\times\RR^d)$ and assume that~$f$ satisfies Assumption~\ref{assu:coefs} (iia) if~$\underline{\kappa_g}=0$ and (iib) if~$\underline{\kappa_g}>0$. Let $p\ge1$, $t,\tb\in[0,T]$, $\bx,\bxb\in  C^0 ([0,\Tb], \mathbb{R}^d)$. There exist $C,\kappa>1$ depending on $C_g,\overline{\kappa_g},\underline{\kappa_g},C_h,c_h,\overline{\kappa_h},\underline{\kappa_h},p,\norm{K}_2,\Tb,$ such that, for all $s \in[0,\Tb]$, 
    \begin{align}
    \EE\left[\abs{h_s(\Wh_s^{t,\bx}) }^p\right]
    &\le C \big(1+ \E^{\kappa\bx_s}\big);  \label{eq:Vol_bound}\\
    \EE\left[\sup_{s\in[0,\Tb]} \abs{h_s(\Wh_s^{t,\bx}) }^p\right]
    &\le C \big(1+ \E^{\kappa\norm{\bx}_\infty}\big);  \label{eq:Vol_bound_infty}\\
     \EE\left[\abs{g(V^{t,\bx}_T)}^p\right]
    &\le  C\left(1+ \E^{\kappa\norminfty{\bx}}\right).\label{eq:Payoff_bound} 
    \end{align}
    If moreover Assumption~\ref{assu:kernel} holds, $g\in C^1_{{\rm poly}}((0,+\infty))$ and $h\in C^{0,1}_{{\rm exp}+}([0,\Tb]\times\RR^d)$, then there exist $C,\kappa>1$ depending on $C_g,\overline{\kappa_g},\underline{\kappa_g},C_h,c_h,\overline{\kappa_h},\underline{\kappa_h},p,\norm{K}_2,\Tb,$ such that, for all $t\le s \le \Tb$, 
    \begin{align}
     \EE\left[\abs{ h_s(\Wh_s^{t,\bx})-h_s(\Wh_s^{t,\bxb})}^p\right] 
    &\le C\abs{\bx_s-\bxb_s}^p (1+\E^{\kappa\bx_s}+\E^{\kappa\bxb_s}), \label{eq:Vol_limit}\\
    \EE\left[\abs{g(V^{t,\bx}_T) - g(V^{t,\bxb}_T)}^p\right] 
    &\le C\norminfty{\bx-\bxb}^p (1+ \E^{\kappa \norminfty{\bxb}} + \E^{\kappa\norminfty{\bx}}); \label{eq:Payoff_limit} \\
    \EE\left[\abs{g(V^{t,\bx}_T) - g(V^{t,\bxb}_T)}^p\right] 
    &\le C  \norm{\bx-\bxb}_{L^1[t,\Tb]}^p \,(1+ \E^{\kappa \norminfty{\bxb}} + \E^{\kappa\norminfty{\bx}});
    \label{eq:Payoff_L2}\\
    \EE\left[\abs{g(V^{t,\bx}_T) - g(V^{\tb,\bx}_T)}^p\right] 
    &\le C \abs{t-\tb}^{Hp} (1+\E^{\kappa\norminfty{\bx}}). \label{eq:Payoff_time_ref}
    \end{align}
\end{lemma}
In particular, 
\begin{itemize}
    \item \eqref{eq:Vol_bound} holds for $\displaystyle h\in\left\{f,\frac{\partial f}{\partial x_i},\frac{\partial^2 f}{\partial x_i\partial x_j},\frac{\partial^3 f}{\partial x_i \partial x_j \partial x_k}: 1\le i,j,k\le d \right\}$;
    \item \eqref{eq:Payoff_bound} holds for $g\in\{\phi,\phi',\phi'',\phi'''\}$;
    \item \eqref{eq:Vol_limit} holds for $\displaystyle h\in\left\{f,\frac{\partial f}{\partial x_i},\frac{\partial^2 f}{\partial x_i\partial x_j}: 1\le i,j\le d \right\}$;
    \item \eqref{eq:Payoff_limit}, \eqref{eq:Payoff_L2} and \eqref{eq:Payoff_time_ref} hold for $g\in\{\phi,\phi',\phi''\}$.
\end{itemize}
\begin{remark}
    Assumption~\ref{assu:kernel} is only used to prove~\eqref{eq:Payoff_time_ref}.
\end{remark} 
\begin{remark}
    In the case (iib), the exponential decay of~$f$ guarantees that~$\EE\left[\abs{V^{t,\bx}_T}^{-q}\right]\le C \E^{\kappa \norminfty{\bx}}$ for all~$q>0$, see~\eqref{eq:Bound_Inverse_Poly}. This bound recovers the setting of polynomial growth of~$g$ and allows to deal with both cases in a unified manner. It is inspired from~\cite[Lemma 6.14]{jacquier2021rough}.
\end{remark}


\begin{proof} 
Throughout the proof, constants $C,\kappa\ge1$ can change from line to line.

    {\bf (1)} Thanks to the exponential growth of $h$ and the Gaussianity of $I^t_s$, we have, for all $s\ge t$,
    \begin{align} \label{eq:first_bound_h}
    \EE\left[\abs{h_s(\Wh_s^{t,\bx}) }^p\right]
    \le C_h^p\EE\left[1+\E^{p\kappa_h(s+\bx_s+I^t_s)}\right]
    = C_h^p \Big(1+ \E^{p^2\kappa_h^2 \EE[(I^t_s)^2] + p\kappa_h (s+\bx_s)} \Big),
    \end{align}
    where $\EE[(I^t_s)^2]=\int_t^s \abs{K(s,r)}^2\dr \le \norm{K}_2$. In the other case $s< t$, $\Wh^{t,\bx}_s=\bx_s$ hence the exponential growth of $h$ yields \eqref{eq:Vol_bound}.

    {\bf (2)} In the same manner as~\eqref{eq:first_bound_h}, we have, for all $s\ge t$,
\begin{align*}
    \EE\left[\sup_{s\in[0,\Tb]} \abs{h_s(\Wh_s^{t,\bx}) }^p\right]
    \le C\left(1+ \E^{p\kappa_h (\Tb+\norm{\bx}_\infty)} \EE\left[\E^{p\kappa_h \norm{I^t}_\infty} \right]\right),
\end{align*}
and the claim follows since ~$I^t$ is a Gaussian process. The other case $s< t$ is the same as {\bf (1)}.
    
    {\bf (3)} In the case~$\underline{\kappa_g}=0$, by the polynomial growth of $g$ and Jensen's inequality
    \begin{align*}
        \EE\left[\abs{g(V^{t,\bx}_T)}^p\right]
        \le C_g^p \EE\left[ 1+ \abs{V^{t,\bx}_T}^{p\kappa_g }\right]
        \le C_g^p \Big(1+ \Tb^{p\kappa_g} \sup_{s\in[0,\Tb]} \EE\left[\abs{f_s(\Wh_s^{t,\bx}) }^{p\kappa_g}\right]\Big),
    \end{align*}
    which yields \eqref{eq:Payoff_bound} by applying \eqref{eq:Vol_bound}. If~$\underline{\kappa_g}>0$ then
    \begin{align*}
        \EE\left[\abs{g(V^{t,\bx}_T)}^p\right]
        \le C_g^p \EE\left[ 1+ \abs{V^{t,\bx}_T}^{p\overline{\kappa_g}}+\abs{V^{t,\bx}_T}^{-p\underline{\kappa_g}}\right].
    \end{align*}
    The first term is identical to the case~$\underline{\kappa_g}=0$. For the second one, we set~$q=p\underline{\kappa_g}>0$, use the exponential decay property and exploit  Jensen's inequality with the concavity of the logarithm to obtain
    \begin{align}\label{eq:Bound_Inverse_Poly}
        &\EE\left[\abs{V^{t,\bx}_T}^{-q}\right] 
        =\EE\left[\E^{-q \ln(V^{t,\bx}_T)}\right]
        \le \EE\left[\exp\left(-\frac{q}{\abs{\Dd}}\EE_T\int_{\Dd}\ln\Big(f\big(\Wh_r^{t,\bx}\big)\Big)\dr\right)\right] \\
        &\le  \EE\left[\exp\left(-\frac{q}{\abs{\Dd}}\EE_T \int_{\Dd}\left\{\ln(c_h\E^{-\underline{\kappa_f} r})+\ln\Big(\sum_{i=1}^d \E^{-\underline{\kappa_h} (\Wh^{t,\bx}_r)^{(i)} }\Big)\right\}\dr\right)\right] \nonumber\\
        &\le C \EE\left[\exp\left(-\frac{q}{\abs{\Dd}}\EE_T\int_{\Dd} \frac{1}{d}\sum_{i=1}^d\ln\Big(d \E^{-\underline{\kappa_h} (\Wh^{t,\bx}_r)^{(i)}}\Big)\dr\right)\right]  \nonumber \\
        &\le C \E^{ \frac{q\underline{\kappa_h}}{ d} \norminfty{\bx}}
        \EE\left[\exp\left(\frac{q\underline{\kappa_h}}{\abs{\Dd} d}\int_{\Dd} \sum_{i=1}^d B_r^{(i)} \dr\right)\right] \le C \E^{\kappa \norminfty{\bx}}, \nonumber
    \end{align}
    where~$\Big(B_r := (B_r^{(i)})_{i=1}^d := \big(\sum_{j=1}^m \int_t^{T\wedge r} K_{ij}(r,v)\D W_v^j\big)_{i=1}^d \Big)_{r\in\Dd}$ is a continuous Gaussian process and thus the expectation on the last line is bounded by~\cite[Lemma 6.13]{jacquier2021rough}. This concludes the proof of~\eqref{eq:Payoff_bound}.
    
    {\bf (4)} By Taylor's theorem with integral remainder, we have
    \begin{align}
        h_s(\Wh_s^{t,\bxb})-h_s(\Wh_s^{t,\bx})
        = \int_0^1 \Big\langle \nabla h_s \big( \lambda \Wh_s^{t,\bxb} +(1-\lambda)\Wh_s^{t,\bx} \big),
        (\Wh_s^{t,\bxb}-\Wh_s^{t,\bx})\Big\rangle  \D\lambda.
    \end{align}
    Notice that $\Wh_s^{t,\bxb}-\Wh_s^{t,\bx}=\bxb_s-\bx_s$, and for all $\lambda\in[0,1]$, $ \lambda \Wh_s^{t,\bxb} +(1-\lambda)\Wh_s^{t,\bx} = \Wh_s^{t,\lambda\bxb+(1-\lambda)\bxb}$.
    Hence, by Cauchy-Schwarz and Jensen's inequalities and \eqref{eq:Vol_bound},
    \begin{align}\label{eq:h_Taylor}
         \EE\left[\abs{ h_s(\Wh_s^{t,\bxb})-h_s(\Wh_s^{t,\bx})}^p\right] 
         &\le \abs{\bxb_s-\bx_s}^p \int_0^1 \EE\abs{\nabla h_s(\Wh_s^{t,\lambda\bxb+(1-\lambda)\bx})}^p \D\lambda\\
         &\le \abs{\bxb_s-\bx_s}^p \int_0^1 C(1+\E^{\kappa(\lambda\bxb_s+(1-\lambda)\bx_s)}) \D\lambda \nonumber\\
         &\le C\abs{\bx_s-\bxb_s}^p (1+\E^{\kappa\bx_s}+\E^{\kappa\bxb_s}),\nonumber
    \end{align}
    which yields \eqref{eq:Vol_limit}.

    {\bf (5)} By Taylor's theorem with integral remainder, we have
    \begin{align}\label{eq:Taylortog}
        g(V^{t,\bxb}_T) - g(V^{t,\bx}_T) = (V^{t,\bxb}_T-V^{t,\bx}_T)  \int_0^1 g'(\lambda V^{t,\bxb}_T +(1-\lambda) V^{t,\bx}_T ) \D\lambda.
    \end{align}
    For all~$p>1$, Jensen's inequality and \eqref{eq:Vol_limit} yield
    \begin{align*}
\EE\abs{V^{t,\bx}_T-V^{t,\bxb}_T}^{p} 
\le \Tb^{p-1} \sup_{s\in[0,\Tb]}  \EE\left[\abs{ f_s(\Wh_s^{t,\bx})-f_s(\Wh_s^{t,\bxb})}^{p}\right] 
\le \Tb^{p-1} C \norminfty{\bx-\bxb}^{p} (1+\E^{\kappa\norminfty{\bx}}+\E^{\kappa\norminfty{\bxb}}).
    \end{align*} 
    As~$g\in C^1_{{\rm poly}}$, the polynomial growth entails for all~$p>1$,
    \begin{align*}
        \EE\abs{g'(\lambda V^{t,\bxb}_T +(1-\lambda) V^{t,\bx}_T )}^{p} 
        &\le C\Big(1+ \EE\abs{V^{t,\bxb}_T }^{p\overline{\kappa_g}} + \EE\abs{V^{t,\bx}_T }^{p\overline{\kappa_g}} + \EE\abs{\lambda V^{t,\bxb}_T +(1-\lambda) V^{t,\bx}_T }^{-\underline{\kappa_g} p} \Big) .
    \end{align*}
    Note that~$\lambda V^{t,\bxb}_T +(1-\lambda) V^{t,\bx}_T = \int_\Dd \EE_T\big[\lambda f(\Wh^{t,\bxb}_r) + (1-\lambda) f(\Wh^{t,\bx}_r) \big] \dr$.
    In the case~$\underline{\kappa_g}>0$, the exponential decay assumption allows to derive the following bound with the same arguments as~\eqref{eq:Bound_Inverse_Poly} for any~$q>0$:
    \begin{align*}
        \EE\left[\abs{\lambda V^{t,\bxb}_T +(1-\lambda) V^{t,\bx}_T }^{-q}\right]
        &\le \EE\left[ \exp\left(-\frac{q}{\abs{\Dd}} \int_\Dd \EE_T\big[\lambda \ln(f(\Wh^{t,\bxb}_r)) + (1-\lambda) \ln(f(\Wh^{t,\bx}_r)) \big] \dr \right)\right] \\
        &\le C \E^{\kappa \norminfty{\bxb} +\kappa \norminfty{\bx}}\le C \left(\E^{\kappa \norminfty{\bxb}} +  \E^{\kappa \norminfty{\bx}} \right). 
    \end{align*}
    In virtue of this inequality and~\eqref{eq:Payoff_bound}, we get~$\EE\abs{g'(\lambda V^{t,\bxb}_T +(1-\lambda) V^{t,\bx}_T )}^{p} \le  C\left(1+ \E^{\kappa\norminfty{\bxb}}+\E^{\kappa\norminfty{\bx}}\right)$.
    The claim \eqref{eq:Payoff_limit} follows after an application of Cauchy-Schwarz inequality to the $L^p$ norm of~\eqref{eq:Taylortog}.

{\bf (6)} Similarly to \eqref{eq:h_Taylor}, we have by Taylor's theorem
    \begin{align*}
\abs{V^{t,\bx}_T-V^{t,\bxb}_T} 
&\le \EE_T \int_0^{\Tb} \abs{\bx_s-\bxb_s} \left(\int_0^1 \abs{\nabla f_s(\Wh_s^{t,\lambda\bxb+(1-\lambda)\bx})}\D \lambda \right)\ds \\
&\le \int_0^{\Tb} \abs{\bx_s-\bxb_s}\ds\,  \EE_T \left[\sup_{s\in[0,\Tb]} \int_0^1 \abs{\nabla f_s(\Wh_s^{t,\lambda\bxb+(1-\lambda)\bx})}\D \lambda\right].
    \end{align*}   
Therefore, in virtue of \eqref{eq:Vol_bound_infty} we obtain
\begin{align*}
    \EE\left[\abs{V^{t,\bx}_T-V^{t,\bxb}_T}^p\right] 
    \le C \left( \int_0^{\Tb} \abs{\bx_s-\bxb_s}\ds\right)^p \,(1+ \E^{\kappa \norminfty{\bxb}} + \E^{\kappa\norminfty{\bx}}).
\end{align*}
Equation \eqref{eq:Taylortog} and the ensuing computations yield the claim.

{\bf (7)} We consider two time points $0\le \tb<t\le \Tb$ and, similarly as above, using Taylor's theorem, Jensen and Cauchy-Schwarz inequalities,
\begin{align}
    \EE\left[\abs{ h_s(\Wh_s^{t,\bx})-h_s(\Wh_s^{\tb,\bx})}^p\right]^2 
         &\le \EE\left[\int_0^1 \abs{\nabla h_s\big( \lambda \Wh^{t,\bx}_s +(1-\lambda) \Wh^{\tb,\bx}_s \big)}^{2p} \D\lambda \right] \EE\left[\abs{\Wh_s^{t,\bx}-\Wh_s^{\tb,\bx}}^{2p} \right] \nonumber\\
         &\le C\big(1+\E^{\kappa\bx_s} \big) \EE\left[\abs{\int_{\tb}^t K(s,r) \D W_r}^{2p}  \right].
\label{eq:vol_time_reg}
\end{align}
Exploiting BDG inequality and Assumption \ref{assu:kernel} we have 
\begin{align*}
    \EE\abs{\int_{\tb}^t K(s,r) \D W_r}^{2p} 
    \le C \left(\int_{\tb}^t \abs{K(s,r)}^2 \dr \right)^p
    \le C \left(\int_{\tb}^t (s-r)^{2H-1} \dr \right)^p
    \le C (t-\tb)^{2Hp},
\end{align*}
where $C>0$ is a constant that changes from inequality to inequality. In the case $2H\le1$ we use that $\RR_+\ni x\mapsto(x+t-\tb)^{2H}-x^{2H}$ is decreasing and hence $(s-\tb)^{2H}-(s-t)^{2H} \le (t-\tb)^{2H}$. If $2H\ge1$ then~$(s-\tb)^{2H}-(s-t)^{2H} \le 2H \Tb^{2H-1} (t-\tb)$ since~$x\mapsto (s-x)^{2H}$ is continuously differentiable.

Again by Taylor's theorem, we have 
$$
g(V^{t,\bx}_T) - g(V^{\tb,\bx}_T) = (V^{t,\bx}_T-V^{\tb,\bx}_T) \int_0^1 g'(\lambda V^{\tb,\bx}_T +(1-\lambda) V^{t,\bx}_T ) \D\lambda.
$$
For all~$p>1$, the integral is bounded in~$L^p$ by similar arguments as item \textbf{(4)},
and by \eqref{eq:vol_time_reg} we have
    \begin{align*}
\EE\abs{V^{t,\bx}_T-V^{\tb,\bx}_T}^{p} 
\le \Tb \sup_{s\in[0,\Tb]}  \EE\left[\abs{ f_s(\Wh_s^{t,\bx})-f_s(\Wh_s^{\tb,\bx})}^{p}\right]\le C \left(1+\E^{\kappa\norminfty{\bx}}\right)(t-\tb)^{Hp}.
\end{align*}
Applying Cauchy-Schwarz and \eqref{eq:Payoff_bound} yields
\begin{align*}
    \EE\abs{g(V^{t,\bx}_T) - g(V^{\tb,\bx}_T)}^p 
    \le C\left(1+\E^{\kappa\norminfty{\bx}}\right)(t-\tb)^{Hp},
\end{align*}
which yields the claim.    
\end{proof}

\begin{lemma}\label{lemma:Limit_payoff}
Let $g\in C^1_{{\rm poly}}(\RR)$ and assume that~$f$ satisfies Assumption~\ref{assu:coefs} (iia) if~$\underline{\kappa_g}=0$ and (iib) if~$\underline{\kappa_g}>0$. Let $p>1$, $t\in[0,T]$, $\bx,\bxb\in  C^0 ([0,\Tb], \mathbb{R}^d)$ and for any $\eta\in \Ww_t$ define $\widetilde{g}(\eta):= \int_0^1 g\big(\lambda V_T^{t,\bx+\eta}+(1-\lambda) V_T^{t,\bx}\big)\D \lambda$. Then we have
\begin{align}
    \lim_{\norminfty{\eta}\to0} \frac{\EE\abs{\widetilde{g}(\eta)-g(V_T^{t,\bx})}^p}{\norminfty{\eta}} &= 0.
    \label{eq:Tilde_limit}
\end{align}
In particular, this holds for $g\in\{\phi',\phi''\}$.
\end{lemma}
\begin{proof}
By Jensen's and Cauchy-Schwarz inequalities,
\begin{align*}
    \EE\abs{\widetilde{g}(\eta) - g(V_T^{t,\bx})}^p 
    &= \EE\abs{\int_0^1 g( V_T^{t,\bx}+\lambda(V_T^{t,\bx+\eta}-V_T^{t,\bx}))-g(V_T^{t,\bx}) \D \lambda}^p \\
    & = \EE\abs{\int_0^1  \lambda(V^{t,\bx+\eta}_T-V^{t,\bx}_T) \int_0^1 g'\Big( V^{t,\bx}_T +\overline{\lambda} \lambda \big(V^{t,\bx+\eta}_T-V^{t,\bx}_T)\big) \D\overline{\lambda}\,  \D\lambda }^p\\
    & \le \EE\abs{V^{t,\bx+\eta}_T-V^{t,\bx}_T}^p \int_0^1 \int_0^1 \abs{g'\Big( V^{t,\bx}_T +\overline{\lambda} \lambda \big(V^{t,\bx+\eta}_T-V^{t,\bx}_T)\big) }^p \D\overline{\lambda}\,  \D\lambda \\
    & \le \left(\EE\abs{V^{t,\bx+\eta}_T-V^{t,\bx}_T}^{2p} \int_0^1 \int_0^1 \EE\abs{g'\Big( V^{t,\bx}_T +\overline{\lambda} \lambda \big(V^{t,\bx+\eta}_T-V^{t,\bx}_T)\big) }^{2p} \D\overline{\lambda}\,  \D\lambda\right)^{1/2}.
\end{align*}
The polynomial growth of $g'$ combined with \eqref{eq:Payoff_bound} and~\eqref{eq:Bound_Inverse_Poly} yield that the second factor is bounded for any~$\underline{\kappa_g}\ge0$. Meanwhile, thanks to \eqref{eq:Payoff_limit}, there exists $C,\kappa>1$ independent of $\eta$ such that
$$
\EE\abs{V^{t,\bx+\eta}_T-V^{t,\bx}_T}^{2p} 
\le  C \norminfty{\eta}^{2p} \Big(1+ \E^{\kappa(\norminfty{\bx}+ \norminfty{\eta})} + \E^{\kappa\norminfty{\bx}} \Big),
$$
which yields the claim as~$p>1$.
\end{proof}


\subsection{Proof of Lemma \ref{lemma:Czerotwo}}
In order to prove that $u\in \Cc^{0,2}_{\half}(\Lambda)$, we first extend $u$ to $\widetilde{\Lambda}$ using the same definition~\eqref{eq:Def_u}. We then need to check that the first and second Fréchet derivatives exist and satisfy the three items of Definition \ref{def:Cplusalpha}. The latter are proved in Lemmas \ref{lemma:DerOne} and \ref{lemma:DerTwo} for the first and second derivatives respectively.
\begin{lemma}\label{lemma:DerOne}
Let Assumptions \ref{assu:kernel} and \ref{assu:coefs} hold. Let $t\in[0,T]$, $\bx,\bxb\in  C^0 ([0,\Tb], \mathbb{R}^d)$ and $\eta\in  \Ww_t$. Then we have
\begin{align}\label{eq:DerOne}
    \langle D_{\bx} u(t,\bx) ,\eta\rangle =\EE\left[\phi'(V^{t,\bx}_T) \int_t^{\Tb} \Big\langle \nabla f_s(\Wh_s^{t,\bx}), \eta_s \Big\rangle \ds\right].
\end{align}
Moreover, if $\eta$ is supported on~$[t,t+\delta]$ for some~$0\le\delta\le \Tb-t$ then $D_\bx u$ satisfies the estimates \eqref{eq:DerOne_bound} and \eqref{eq:DerOne_reg} with~$\alpha=\half$, $\varrho={\rm id}$ and is continuous in $t$.
\end{lemma}
\begin{proof}
$\bm{(D)}$ 
Let $t\in[0,T]$, $\bx,\bxb\in  C^0 ([0,\Tb], \mathbb{R}^d)$ and $\eta\in  \Ww_t$. 
Recall that, by the definition of the Fréchet derivative~\eqref{eq:PathwiseDerDef}, we need to prove the convergence
$$
\lim_{\norminfty{\eta}\to0} \frac{u(t,\bx+\eta)-u(t,\bx)-\EE\left[\phi'(V^{t,\bx}_T) \int_t^{\Tb} \Big\langle \nabla f_s(\Wh_s^{t,\bx}), \eta_s \Big\rangle \ds\right]}{\norminfty{\eta}} = 0.
$$
We write
\begin{align*}
    \phi(V^{t,\bx+\eta}_T) - \phi(V_T^{t,\bx}) = (V^{t,\bx+\eta}_T - V_T^{t,\bx}) \widetilde{\phi'}(\eta),
\end{align*}
where $\widetilde{\phi'}$ is defined as in Lemma \ref{lemma:Limit_payoff}.
We consider
\begin{align}\label{eq:Decomposition_FD}
    &\phi(V^{t,\bx+\eta}_T) - \phi(V_T^{t,\bx}) - \phi'(V^{t,\bx}_T) \int_t^{\Tb} \Big\langle \nabla f_s(\Wh_s^{t,\bx}), \eta_s \Big\rangle\ds \\
    &= \left\{V^{t,\bx+\eta}_T - V_T^{t,\bx}- \int_t^{\Tb}   \Big\langle \nabla f_s(\Wh_s^{t,\bx}), \eta_s \Big\rangle\ds \right\} \widetilde{\phi'}(\eta) + \left\{\widetilde{\phi'}(\eta) -\phi'(V_T^{t,\bx})\right\} \int_t^{\Tb}   \Big\langle \nabla f_s(\Wh_s^{t,\bx}), \eta_s \Big\rangle\ds. \nonumber
\end{align}
By Cauchy-Schwarz inequality and the bound \eqref{eq:Vol_bound} there is~$C,\kappa>1$ such that~$\EE\abs{\int_t^{\Tb}   \Big\langle \nabla f_s(\Wh_s^{t,\bx}), \eta_s \Big\rangle\ds}^2 \le C(1+\E^{\kappa\norminfty{\bx}}) \norminfty{\eta}^2$. Hence, Cauchy-Schwarz inequality yields
$$
\left(\frac{1}{\norminfty{\eta}}\EE\,\Big\lvert\!\left\{\widetilde{\phi'}(\eta) -\phi'(V_T^{t,\bx})\right\} \int_t^{\Tb}   \Big\langle \nabla f_s(\Wh_s^{t,\bx}), \eta_s \Big\rangle\ds\Big\lvert \right)^2
\le C(1+\E^{\kappa\norminfty{\bx}})\EE\abs{\widetilde{\phi'}(\eta) -\phi'(V_T^{t,\bx})}^2
$$
which tends to zero as~$\norminfty{\eta}$ goes to zero thanks to  the limit \eqref{eq:Tilde_limit} applied to $\phi'$.
Regarding the first term of \eqref{eq:Decomposition_FD}, first notice that~$\EE\big\lvert\widetilde{\phi'}(\eta)\big\lvert$ is bounded by the growth assumption of~$\phi'$ and the bound~\eqref{eq:Payoff_bound}. We have
\begin{align*}
    V^{t,\bx+\eta}_T - V_T^{t,\bx}
    &=\int_t^{\Tb} \EE_T\big[f_s(\Wh^{t,\bx+\eta}_s)-f_s(\Wh^{t,\bx}_s)\big] \ds \\
    &=\int_t^{\Tb} \EE_T\left[ 
    \Big\langle \nabla f_s(\lambda \Wh_T^{t,\bx+\eta} + (1-\lambda) \Wh^{t,\bx}_T), \Wh^{t,\bx+\eta}_s-\Wh^{t,\bx}_s\Big\rangle  \right] \ds,
\end{align*}
with $\lambda \Wh_T^{t,\bx+\eta} + (1-\lambda) \Wh^{t,\bx}_T = \Wh^{t,\bx+\lambda \eta}_T$ and $\Wh^{t,\bx+\eta}_s-\Wh^{t,\bx}_s=\eta_s$.
Hence Cauchy-Schwarz and Jensen's inequalities yield
\begin{align*}
    \frac{1}{\norm{\eta}_{[t,\Tb]}^2} \EE\left[\abs{V^{t,\bx+\eta}_T - V_T^{t,\bx}- \int_0^{\Tb}  \Big\langle \nabla f_s(\Wh_s^{t,\bx}), \eta_s \Big\rangle \ds}^2\right]
&= \frac{1}{\norm{\eta}_{[t,\Tb]}^2} \EE\abs{\int_0^{\Tb} 
 \Big\langle \nabla f_s(\Wh_s^{t,\bx+\lambda\eta}) - \nabla f_s(\Wh_s^{t,\bx}), \eta_s \Big\rangle
\ds  }^2 \\
&\le \int_t^{\Tb}  \int_0^1 \EE\left[\abs{\nabla f_s(\Wh^{t,\bx+ \lambda \eta}_s ) - \nabla f_s(\Wh^{t,\bx}_s ) }^2\right]\D \lambda \ds,
\end{align*}
which goes to zero as $\norminfty{\eta}$ goes to zero by \eqref{eq:Vol_limit} and proves \eqref{eq:DerOne}.

{\bf (i)} Let $\eta$ be supported on~$[t,t+\delta]$ from now on. By Cauchy-Schwarz and Jensen's inequalities, as well as estimates \eqref{eq:Payoff_bound} and \eqref{eq:Vol_bound}, we have
\begin{align*}
    \abs{\EE\left[\phi'(V^{t,\bx}_T) \int_t^{\Tb} \Big\langle \nabla f_s(\Wh_s^{t,\bx}), \eta_s \Big\rangle\ds\right]}^2
    &\le \EE\left[\abs{\phi'(V^{t,\bx}_T)}^2\right] \, \delta \int_t^{t+\delta} \EE \left[\abs{\Big\langle \nabla f_s(\Wh_s^{t,\bx}), \eta_s \Big\rangle}^2\right] \ds\\
    &\le \sup_{s\in[0,\Tb]} \left\{ C\Big( 1+ \E^{\kappa\bx_s}\Big) \delta  \EE\left[\abs{\nabla f_s(\Wh_s^{t,\bx})}^2\right]\right\} \int_t^{t+\delta} \abs{\eta_s}^2 \ds \\
    &\le C\Big( 1+ \E^{\kappa\norminfty{\bx}}\Big) \delta^2 \norm{\eta}_\infty^2,
\end{align*}
where the constants $C,\kappa>1$ are independent of $\bx,\eta$ and can change from line to line.
This proves~\eqref{eq:DerOne_bound} with $\alpha=1$. 

{\bf (ii)} Yet again the same inequalities entail
\begin{align*}
     &\abs{\langle D_\bx u(t,\bx) ,\eta\rangle -  \langle D_\bx u(t,\bxb) ,\eta\rangle }\\
     &= \abs{\EE\left[\left\{\phi'(V^{t,\bx}_T) - \phi'(V^{t,\bxb}_T) \right\}\int_t^{t+\delta} \Big\langle \nabla f_s(\Wh_s^{t,\bx}), \eta_s \Big\rangle \ds\right] 
     + \EE\left[\int_t^{t+\delta} \Big\langle \nabla f_s(\Wh_s^{t,\bx})-\nabla f_s(\Wh_s^{t,\bxb}), \eta_s\Big\rangle \ds \,\phi'(V^{t,\bxb}_T)\right] }\\
     &\le \bigg(\EE \abs{\phi'(V^{t,\bx}_T) - \phi'(V^{t,\bxb}_T)}^2 \delta \sup_{s\in[0,\Tb]} \EE\abs{\nabla f_s(\Wh_s^{t,\bx})}^2 \int_t^{t+\delta} \abs{\eta_s}^2\ds\\
     &\qquad+ \delta\sup_{s\in[0,\Tb]}\EE\abs{\nabla f_s(\Wh_s^{t,\bx})-\nabla f_s(\Wh_s^{t,\bxb})}^2\int_t^{t+\delta}\abs{\eta_s}^2\ds \, \EE\abs{\phi'(V^{t,\bxb}_T)}^2
     \bigg)^{1/2}.
\end{align*}
Combining the first four estimates of Lemma \ref{lemma:Estimates_fg} yield \eqref{eq:DerOne_reg} with $\alpha=1$ and $\varrho={\rm id}$.

{\bf (iii)} The time regularity unfolds as in {\bf (ii)} albeit using the estimates \eqref{eq:vol_time_reg} and \eqref{eq:Payoff_time_ref}.
\end{proof}


\begin{lemma}\label{lemma:DerTwo}
Let Assumptions \ref{assu:kernel} and \ref{assu:coefs} hold.
Let $t\in[0,T]$, $\bx,\bxb\in  C^0 ([0,\Tb], \mathbb{R}^d)$ and $\eta,\etab\in  \Ww_t$. Then we have
\begin{align}\label{eq:Der_two_u}
    \langle D_\bx^2 u(t,\bx) ,(\eta,\etab)\rangle
    &=\EE\left[\phi''(V^{t,\bx}_T) \int_t^{\Tb} \Big\langle\nabla f_s(\Wh_s^{t,\bx}), \eta_s \Big\rangle\ds \,\int_t^{\Tb} \Big\langle\nabla f_s(\Wh_s^{t,\bx}), \etab_s \Big\rangle\ds\right]\\
    &\quad+\EE\left[ \phi'(V^{t,\bx}_T) \int_t^{\Tb} \Big\langle\nabla^2 f_s(\Wh_s^{t,\bx}), (\eta_s ,\etab_s)\Big\rangle\ds 
    \right].
    \nonumber
\end{align}
If, moreover, $\eta$ and $\etab$ are supported on~$[t,t+\delta]$ for some~$0\le\delta\le \Tb-t$ then $D_\bx^2 u$ satisfies the estimates~\eqref{eq:DerTwo_bound} and \eqref{eq:DerTwo_reg} with~$\alpha=\half$, $\varrho={\rm id}$ and is continuous in $t$.

Moreover, for any~$\eta,\etab\in\Ww_t$, $t\le \bar{t}\le T$, the following $L^2$ estimates also hold
    \begin{align*}
     &\abs{\big\langle D^2_\bx u(t,\bx),(\eta,\etab)\big\rangle}
     \le C \big(1+\E^{\kappa\norminfty{\bx}}\big) \norm{\eta}_{L^2[t,\Tb]}\norm{\etab}_{L^2[t,\Tb]},\\
     &\abs{\big\langle D^2_\bx u(t,\bx),(\eta,\etab)\big\rangle -\big\langle D^2_\bx u(t,\bar{\bx}),(\eta,\etab)\big\rangle }\\
     &\qquad \le C \big(1+\E^{\kappa\norminfty{\bx}}+\E^{\kappa\norminfty{\bxb}}\big) \left(\norm{(\bx-\bxb)}_{L^2[t,\Tb]} \norm{\eta}_{L^2[t,\Tb]} \norm{\etab}_{L^2[t,\Tb]} + \norm{(\bx-\bxb)\eta\,\etab}_{L^1[t,\Tb]} \right),\\
     &\abs{\big\langle D^2_\bx u(t,\bx),(\eta,\etab)\big\rangle -\big\langle D^2_\bx u(\bar{t},\bx),(\eta,\etab)\big\rangle }\\
     &\qquad\le C  \big(1+\E^{\kappa\norminfty{\bx}}\big) \Big(\abs{t-\bar{t}}^H \norm{\eta}_{L^2[t,\Tb]}\norm{\etab}_{L^2[t,\Tb]} 
     +\norm{\etab}_{L^2[t,\Tb]} \norm{\eta}_{L^2[t,\bar{t}]} + \norm{\eta}_{L^2[t,\Tb]} \norm{\etab}_{L^2[t,\bar{t}]}\Big).
    \end{align*}
\end{lemma}

\begin{proof}
    $\bm{(D^2)}$
    Let $t\in[0,T]$, $\bx,\bxb\in  C^0 ([0,\Tb], \mathbb{R}^d)$ and $\eta,\etab\in  \Ww_t$. 
    We aim at proving a similar type of convergence as in the proof of Lemma~\ref{lemma:DerOne}.
    Let us start with the observation that
    \begin{align*}
        \langle D_\bx u(t,\bx+\etab) ,\eta\rangle - \langle D_\bx u(t,\bx) ,\eta\rangle
        &= \EE\left[\Big(\phi'(V^{t,\bx+\etab}_T) -\phi'(V^{t,\bx}_T) \Big)\int_t^{\Tb} \Big\langle\nabla f_s(\Wh_s^{t,\bx}), \eta_s \Big\rangle \ds\right]\\
        &\quad + \EE\left[\phi'(V^{t,\bx+\etab}_T) \int_t^{\Tb} \Big\langle\nabla  f_s(\Wh_s^{t,\bx+\etab}) -\nabla f_s(\Wh_s^{t,\bx}) , \eta_s \Big\rangle\ds\right].
    \end{align*}
    We then consider the difference of the first term with the first term of~\eqref{eq:Der_two_u}
    \begin{align*}
        &\EE\left[\left(\phi'(V^{t,\bx+\etab}_T) -\phi'(V^{t,\bx}_T)
        -\phi''(V^{t,\bx}_T) \int_t^{\Tb} \Big\langle\nabla f_s(\Wh_s^{t,\bx}), \etab_s \Big\rangle\ds\right)
        \int_t^{\Tb} \Big\langle\nabla f_s(\Wh_s^{t,\bx}), \eta_s \Big\rangle \ds\right]^2\\
        &\le \EE\left[\left(\phi'(V^{t,\bx+\etab}_T) -\phi'(V^{t,\bx}_T)
        -\phi''(V^{t,\bx}_T) \int_t^{\Tb} \Big\langle\nabla f_s(\Wh_s^{t,\bx}), \etab_s \Big\rangle\ds\right)^2\right] C\Big(1+\E^{\kappa\norminfty{\bx}}\Big) \norm{\eta}_\infty^2,
    \end{align*}
    where we used Cauchy-Schwarz inequality, estimate \eqref{eq:Vol_bound} and $C,\kappa>1$ are independent of $\ep,\bx,\eta$. The analysis performed in \eqref{eq:Decomposition_FD} and below applies similarly here and proves that this quantity is~$o(\norminfty{\etab}^2)$. Turning to the second term we have
    \begin{align*}
    &\EE\left[\phi'(V^{t,\bx+\etab}_T) \int_0^{\Tb} \Big\langle\nabla  f_s(\Wh_s^{t,\bx+\etab}) -\nabla f_s(\Wh_s^{t,\bx}) , \eta_s \Big\rangle\ds\right] - \EE\left[ \phi'(V^{t,\bx}_T) \int_t^{\Tb} \Big\langle\nabla^2 f_s(\Wh_s^{t,\bx}), (\eta_s ,\etab_s)\Big\rangle\ds 
    \right]\\
    &= \EE\left[ \Big(\phi'(V^{t,\bx+\etab}_T)-\phi'(V^{t,\bx}_T) \Big) \int_t^{\Tb} \Big\langle\nabla^2 f_s(\Wh_s^{t,\bx}), (\eta_s ,\etab_s)\Big\rangle\ds 
    \right]\\
    &\quad+ \EE\left[\phi'(V^{t,\bx+\etab}_T)  \int_t^{\Tb}\left(\Big\langle\nabla  f_s(\Wh_s^{t,\bx+\etab}) -\nabla f_s(\Wh_s^{t,\bx}) , \eta_s \Big\rangle - \Big\langle\nabla^2 f_s(\Wh_s^{t,\bx}), (\eta_s ,\etab_s)\Big\rangle\right)\ds\right]\\
    &=: (I) + (II).
    \end{align*}
By Cauchy-Schwarz and Jensen's inequalities we obtain
\begin{align*}
    \abs{(I)}^2 \le \EE\left[ \abs{\phi'(V^{t,\bx}_T) -\phi'(V^{t,\bx+\etab}_T)}^2\right] \Tb\int_t^{\Tb}\EE\left[ \abs{\Big\langle\nabla^2 f_s(\Wh_s^{t,\bx}), (\eta_s ,\etab_s)\Big\rangle}^2 
    \right]\ds,
\end{align*}
where the first expectation is bounded by~$C\norminfty{\etab}^2(1+\E^{\kappa\norminfty{\bx}}+\E^{\kappa\norminfty{\bx+\etab}})$,  for some~$C,\kappa>1$, in virtue of \eqref{eq:Payoff_limit}. Applying Cauchy-Schwarz inequality twice yields
\begin{align*}
    \EE\left[\abs{\Big\langle\nabla^2 f_s(\Wh_s^{t,\bx}), (\eta_s ,\etab_s)\Big\rangle}^2\right]
    \le \EE\left[ \abs{\nabla^2 f_s(\Wh_s^{t,\bx})}^2 \abs{\eta_s} \abs{\etab_s}\right] 
    \le \sum_{i,j=1}^d \EE\left[\abs{\nabla^2 f_s^{(i,j)}(\Wh_s^{t,\bx})}^2 \right]\norm{\eta}^2_\infty\norm{\etab}^2_\infty,
\end{align*}
which uniform bound over $s\in[0,\Tb]$ goes to zero as~$\norminfty{\etab}$ goes to zero, by \eqref{eq:Vol_bound}. Thus~$\abs{(I)}^2=o(\norminfty{\etab}^2)$. Before tackling $(II)$, we notice that
\begin{align*}
    \Big\langle\nabla  f_s(\Wh_s^{t,\bx+\etab}) -\nabla f_s(\Wh_s^{t,\bx}) , \eta_s \Big\rangle 
    &= \int_0^1 \Big\langle\nabla^2  f_s(\lambda \Wh_s^{t,\bx+\etab} + (1-\lambda)\Wh_s^{t,\bx}), \big(\eta_s, \Wh_s^{t,\bx+\etab}-\Wh_s^{t,\bx} \big)\Big\rangle \D\lambda\\
    &=  \int_0^1 \Big\langle\nabla^2  f_s(\Wh_s^{t,\bx+\lambda \etab}, \big(\eta_s, \etab_s \big)\Big\rangle \D\lambda.
\end{align*}
Applying Cauchy-Schwarz and Jensen's inequalities again yield
\begin{align*}
    \abs{(II)}^2 
    \le \EE\abs{\phi'(V^{t,\bx+\etab}_T)}^2  \Tb \int_0^{\Tb}\int_0^1  
    \EE \abs{ \Big\langle\nabla^2  f_s(\Wh_s^{t,\bx})-\nabla^2  f_s(\Wh_s^{t,\bx+\lambda \etab}), \big(\eta_s, \etab_s \big)\Big\rangle}^2\D\lambda\ds,
\end{align*}
where the first term is bounded thanks to \eqref{eq:Payoff_bound} and
\begin{align*}
    \frac{1}{\norm{\etab}_\infty^2}\EE \abs{ \Big\langle\nabla^2  f_s(\Wh_s^{t,\bx})-\nabla^2  f_s(\Wh_s^{t,\bx+\lambda \etab}), \big(\eta_s, \etab_s \big)\Big\rangle}^2
    &\le \frac{1}{\norm{\etab}_\infty^2} \EE \abs{\nabla^2  f_s(\Wh_s^{t,\bx})-\nabla^2  f_s(\Wh_s^{t,\bx+\lambda \etab})}^2 \abs{\eta_s}^2 \abs{\etab_s}^2\\
    &\le \sum_{i,j=1}^d \EE \abs{\nabla^2  f_s^{(i,j)}(\Wh_s^{t,\bx})-\nabla^2  f_s^{(i,j)}(\Wh_s^{t,\bx+\lambda \etab})}^2 \norm{\eta}_\infty^2 ,
\end{align*}
and this goes to zero, uniformly in~$s\in\Tb$, by \eqref{eq:Payoff_limit}. This proves the convergence of the second term and hence concludes the first part of the proof.

{\bf (i)} Let $\eta,\etab$ be supported on~$[t,t+\delta]$ for some~$0\le\delta\le \Tb-t$. Then Hölder's and Jensen's inequalities combined with Estimates~\eqref{eq:Vol_bound} and \eqref{eq:Payoff_bound} show
    \begin{align}\label{eq:Proof_DerTwo_Reg}
        &\abs{\EE\left[\phi''(V^{t,\bx}_T) \int_t^{\Tb} \Big\langle\nabla f_s(\Wh_s^{t,\bx}), \eta_s \Big\rangle\ds \,\int_t^{\Tb} \Big\langle\nabla f_s(\Wh_s^{t,\bx}), \etab_s \Big\rangle\ds\right]}^3\\
        &\le \EE\abs{\phi''(V^{t,\bx}_T)}^3 \, \EE\left[ \left( \int_t^{t+\delta} \Big\langle\nabla f_s(\Wh_s^{t,\bx}), \eta_s \Big\rangle\ds\right)^3\right] \,\EE\left[ \left(\int_t^{t+\delta} \Big\langle\nabla f_s(\Wh_s^{t,\bx}), \etab_s \Big\rangle\ds\right)^3\right] \nonumber\\
        &\le \EE\abs{\phi''(V^{t,\bx}_T)}^3 \,
        \delta^2 \int_t^{t+\delta}  \EE\abs{\nabla f_s(\Wh_s^{t,\bx})}^3 \abs{\eta_s}^3 \ds \,
        \delta^2 \int_t^{t+\delta}  \EE\abs{\nabla f_s(\Wh_s^{t,\bx})}^3 \abs{\etab_s}^3 \ds \nonumber\\
        &\le C\big(1+\E^{\kappa\norminfty{\bx}}\big) \delta^6 \norm{\eta}^3_\infty\norm{\etab}^3_\infty, \nonumber
    \end{align}
    for some $C,\kappa>1$ independent of $\bx,\eta,\etab$. Similarly,
    \begin{align*}
        \EE\left[ \phi'(V^{t,\bx}_T) \int_t^{\Tb} \Big\langle\nabla^2 f_s(\Wh_s^{t,\bx}), (\eta_s ,\etab_s)\Big\rangle\ds 
    \right]^2
    &\le \EE\abs{\phi'(V^{t,\bx}_T)}^2 \delta \int_t^{t+\delta} \EE\abs{\nabla^2 f_s(\Wh_s^{t,\bx})}^2 \abs{\eta_s}^2 \abs{\etab_s}^2 \ds\\
    &\le C\big(1+\E^{\kappa\norminfty{\bx}}\big)\delta^2 \norm{\eta}^2_\infty\norm{\etab}^2_\infty,
    \end{align*}
    for some $C,\kappa>1$ independent of $\bx,\eta,\etab$, where we used again \eqref{eq:Vol_bound} and \eqref{eq:Payoff_bound} to conclude. This proves that \eqref{eq:DerTwo_bound} is satisfied with $\alpha=1/2$.
    By similar computations as above but using Cauchy-Schwarz inequality we get that for any~$\eta,\etab\in\Ww_t$,
    \begin{align*}
     \abs{\big\langle D^2_\bx u(t,\bx),(\eta,\etab)\big\rangle}
     \le C \big(1+\E^{\kappa\norminfty{\bx}}\big) \norm{\eta}_{L^2[t,\Tb]}\norm{\etab}_{L^2[t,\Tb]}.
    \end{align*}

{\bf (ii)} We look at the regularity of~\eqref{eq:Der_two_u} with $\bx,\bxb\in C^0 ([0,\Tb], \mathbb{R}^d)$ and start with the first term. We split it in three and apply Hölder's inequality:
\begin{align}
    &\Bigg\lvert\EE\left[\phi''(V^{t,\bx}_T) \int_t^{\Tb} \Big\langle\nabla f_s(\Wh_s^{t,\bx}), \eta_s \Big\rangle\ds \,\int_t^{\Tb} \Big\langle\nabla f_s(\Wh_s^{t,\bx}), \etab_s \Big\rangle\ds\right] \nonumber\\
    &\quad- \EE\left[\phi''(V^{t,\bxb}_T) \int_t^{\Tb} \Big\langle\nabla f_s(\Wh_s^{t,\bxb}), \eta_s \Big\rangle\ds \,\int_t^{\Tb} \Big\langle\nabla f_s(\Wh_s^{t,\bxb}), \etab_s \Big\rangle\ds\right]\Bigg\lvert^3 \nonumber\\
    &\le \EE\abs{\phi''(V^{t,\bx}_T)-\phi''(V^{t,\bxb}_T)}^3
    \EE\abs{\int_t^{\Tb} \Big\langle\nabla f_s(\Wh_s^{t,\bx}), \eta_s \Big\rangle\ds}^3 \EE\abs{\int_t^{\Tb} \Big\langle\nabla f_s(\Wh_s^{t,\bx}), \etab_s \Big\rangle\ds}^3 \label{eq:Diff_D2_1} \\
    &\quad + \EE\abs{\phi''(V^{t,\bxb}_T)}^3 \EE\abs{ \int_t^{\Tb} \Big\langle\nabla f_s(\Wh_s^{t,\bx}), \etab_s \Big\rangle\ds}^3
    \EE\abs{\int_t^{\Tb} \Big\langle\nabla f_s(\Wh_s^{t,\bx}), \eta_s \Big\rangle\ds -\int_t^{\Tb} \Big\langle\nabla f_s(\Wh_s^{t,\bxb}), \eta_s \Big\rangle\ds }^3 \label{eq:Diff_D2_2}\\
    &\quad + \EE\abs{\phi''(V^{t,\bxb}_T)}^3 \EE\abs{ \int_t^{\Tb} \Big\langle\nabla f_s(\Wh_s^{t,\bxb}), \eta_s \Big\rangle\ds}^3
    \EE\abs{\int_t^{\Tb} \Big\langle\nabla f_s(\Wh_s^{t,\bx}), \etab_s \Big\rangle\ds -\int_t^{\Tb} \Big\langle\nabla f_s(\Wh_s^{t,\bxb}), \etab_s \Big\rangle\ds }^3 \label{eq:Diff_D2_3}\\
    &\le C\norminfty{\bx-\bxb}^3 \left(1+\E^{\kappa\norminfty{\bx}}+\E^{\kappa\norminfty{\bxb}}\right) \delta^6 \norm{\eta}^3_\infty\norm{\etab}^3_\infty, \nonumber
\end{align}
where we concluded using Jensen's inequality in the same way as in \eqref{eq:Proof_DerTwo_Reg} as well as all the estimates of Lemma \ref{lemma:Estimates_fg}. Similarly, for the second term we apply Cauchy-Schwarz inequality
\begin{align}
    &\abs{\EE\left[ \phi'(V^{t,\bx}_T) \int_t^{\Tb} \Big\langle\nabla^2 f_s(\Wh_s^{t,\bx}), (\eta_s ,\etab_s)\Big\rangle\ds 
    \right]
    - \EE\left[ \phi'(V^{t,\bxb}_T) \int_t^{\Tb} \Big\langle\nabla^2 f_s(\Wh_s^{t,\bxb}), (\eta_s ,\etab_s)\Big\rangle\ds 
    \right]}^2 \nonumber\\
    &\le \EE\abs{\phi'(V^{t,\bx}_T)-\phi'(V^{t,\bxb}_T)}^2 \EE\abs{\int_t^{\Tb} \Big\langle\nabla^2 f_s(\Wh_s^{t,\bx}), (\eta_s ,\etab_s)\Big\rangle\ds }^2 \label{eq:Diff_D2_4}\\
    &\quad + \EE\abs{\phi'(V^{t,\bxb}_T)}^2 \EE\abs{\int_t^{\Tb} \Big\langle\nabla^2 f_s(\Wh_s^{t,\bx}), (\eta_s ,\etab_s)\Big\rangle\ds -\int_t^{\Tb} \Big\langle\nabla^2 f_s(\Wh_s^{t,\bxb}), (\eta_s ,\etab_s)\Big\rangle\ds}^2 \label{eq:Diff_D2_5}\\
    &\le C\norminfty{\bx-\bxb}^2 \left(1+\E^{\kappa\norminfty{\bx}}+\E^{\kappa\norminfty{\bxb}} \right) \delta^2 \norm{\eta}^2_\infty\norm{\etab}^2_\infty,\nonumber
\end{align}
where we again used Lemma \ref{lemma:Estimates_fg} to conclude. This proves that $\partial^2_\bx u$ satisfies \eqref{eq:DerTwo_reg} with $\alpha=1/2$ and $\varrho={\rm id}$. 
By similar computations as above, exploiting the estimates~\eqref{eq:Vol_bound_infty}, \eqref{eq:Payoff_L2} and Cauchy-Schwarz inequality we get that for any~$\eta,\etab\in\Ww_t$, \eqref{eq:Diff_D2_1}, \eqref{eq:Diff_D2_2}, \eqref{eq:Diff_D2_3} and \eqref{eq:Diff_D2_4} are all bounded by
$$
C \norm{(\bx-\bxb)}_{L^2[t,\Tb]} \norm{\eta}_{L^2[t,\Tb]} \norm{\etab}_{L^2[t,\Tb]} \big(1+\E^{\kappa\norminfty{\bx}}+\E^{\kappa\norminfty{\bxb}} \big),
$$
while  \eqref{eq:Diff_D2_4} is bounded by
\begin{align*}
C \norm{(\bx-\bxb)\eta\etab}_{L^1[t,\Tb]} \big(1+\E^{\kappa\norminfty{\bx}}+\E^{\kappa\norminfty{\bxb}} \big).   
\end{align*}
This concludes the proof of this estimate.

{\bf (iii)} The time regularity unfolds as in {\bf (ii)} albeit using the estimates \eqref{eq:vol_time_reg} and \eqref{eq:Payoff_time_ref}. 
More precisely, for any~$\eta,\etab\in\Ww_t$ and~$t\le\bar{t}\le T$ it holds
    \begin{align*}
     \abs{\big\langle D^2_\bx u(t,\bx),(\eta,\etab)\big\rangle -\big\langle D^2_\bx u(\bar{t},\bx),(\eta,\etab)\big\rangle }
     &\le C  \big(1+\E^{\kappa\norminfty{\bx}}\big) \Big(\abs{t-\bar{t}}^H \norm{\eta_s}_{L^2[t,\Tb]}\norm{\etab_s}_{L^2[t,\Tb]} \\
     &\quad+ \norm{\etab}_{L^2[t,\Tb]} \norm{\eta}_{L^2[t,\bar{t}]} + \norm{\eta}_{L^2[t,\Tb]} \norm{\etab}_{L^2[t,\bar{t}]}\Big).
    \end{align*}
    This concludes the proof.
\end{proof}

\subsection{Proof of Proposition \ref{prop:SingularDirections}}\label{sec:SingularDirections}
We remind the reader that derivatives in the singular direction of~$K^t$ are defined in \eqref{eq:Def_SingularDerivatives} as the limits of the derivatives with the truncated kernel $K^{t,\delta}$. Let $\delta>0$ and~$(t,\bx)\in\Lambda$. 

$\bm{(\partial)}$ We name $v(t,\bx)$ the right-hand side of \eqref{eq:DerOne_Singular}. By linearity of the Fréchet derivative, Cauchy-Schwarz and Jensen's inequalities, we have
\begin{align*}
    \abs{\langle D_\bx u(t,\bx) ,K^{\delta,t}\rangle -v(t,\bx) }^2
    &=\abs{\EE\left[\phi'(V^{t,\bx}_T) \int_t^{\Tb}  \Big\langle \nabla f_s(\Wh_s^{t,\bx}), K^\delta(s,t) -K(s,t) \Big\rangle \ds\right]}^2 \\
    &\le \EE\abs{\phi'(V^{t,\bx}_T)}^2 \,\Tb \sup_{s\in[0,\Tb]} \EE \abs{\nabla f_s(\Wh_s^{t,\bx})}^2 \int_t^{\Tb} \abs{K^\delta(s,t) -K(s,t) }^2 \ds.
\end{align*}
Estimates \eqref{eq:Payoff_bound} and \eqref{eq:Vol_bound} ensure that the expectations are uniformly bounded. Moreover, we note that~$K^t$ and~$K^{t,\delta}$ are different only on $[t,t+\delta]$ hence we have
\begin{align}\label{eq:Diff_Kdelta}
    \int_0^{\Tb} \abs{K^\delta(s,t) -K(s,t) }^2\ds
    = \int_t^{t+\delta} \abs{K(t+\delta,t) - K(s,t)}^2\ds,
\end{align}
and, with $s\in[t,t+\delta]$, Assumption \ref{assu:kernel} yields
\begin{align}
    \abs{K(t+\delta,t) - K(s,t)}^2 
    &= \abs{(t+\delta-s) \int_0^1 \partial_s K\big(\lambda(t+\delta) + (1-\lambda) s,t\big) \D\lambda}^2 \label{eq:Kdelta-K}\\
    &\le 4\delta^2 \left(\int_0^1 C\big( \lambda(t+\delta-s) +(s-t)  \big)^{H-3/2} \D\lambda \right)^2 \nonumber\\
    &= \left(\frac{2C\delta}{H-\half} \right)^2\Big[\delta^{H-1/2} - (s-t)^{H-1/2}\Big]^2 \nonumber \\
    &\le \left(\frac{C\delta}{H-\half} \right)^2 (s-t)^{2H-1},\nonumber
\end{align}
where the constant~$C$ depends on~$H$ and may change from line to line. 
Therefore, coming back to~\eqref{eq:Diff_Kdelta}, we obtain
\begin{align*}
    \int_0^{\Tb} \abs{K^\delta(s,t) -K(s,t) }^2\ds
    \le  \left(\frac{C\delta}{H-\half} \right)^2 \frac{\delta^{2H}}{2H},
\end{align*}
which tends to zero as $\delta$ goes to zero.

$\bm{(\partial^2)}$ We call $w(t,\bx)$ the right-hand side of \eqref{eq:DerTwo_Singular}. Once more by linearity of the Fréchet derivative, we have 
\begin{align*}
    &\big\langle D_\bx^2 u(t,\bx) ,(K^{\delta,t},K^{\delta,t})\big\rangle-w(t,\bx) \\
    &=\EE\left[\phi''(V^{t,\bx}_T) \int_t^{\Tb} \Big\langle\nabla f_s(\Wh_s^{t,\bx}), K^\delta(s,t)-K(s,t) \Big\rangle\ds \,\int_t^{\Tb} \Big\langle\nabla f_s(\Wh_s^{t,\bx}), K^\delta(s,t) \Big\rangle\ds\right]\\
    &\quad + \EE\left[\phi''(V^{t,\bx}_T) \int_t^{\Tb} \Big\langle\nabla f_s(\Wh_s^{t,\bx}), K(s,t) \Big\rangle\ds \,\int_t^{\Tb} \Big\langle\nabla f_s(\Wh_s^{t,\bx}), K^\delta(s,t)-K(s,t) \Big\rangle\ds\right]\\
    &\quad+\EE\left[ \phi'(V^{t,\bx}_T) \int_t^{\Tb} \Big\langle\nabla^2 f_s(\Wh_s^{t,\bx}), \big(K^\delta(s,t) ,K^\delta(s,t)-K(s,t)\big)\Big\rangle\ds 
    \right]\\
    &\quad+\EE\left[ \phi'(V^{t,\bx}_T) \int_t^{\Tb} \Big\langle\nabla^2 f_s(\Wh_s^{t,\bx}), \big(K^\delta(s,t)-K(s,t),K(s,t)\big)\Big\rangle\ds 
    \right]\\
    &=: (I) +(II)+(III)+(IV).
\end{align*}
Let $p=\frac{2}{1-H}>2$ and note that $H+2/p=1$, which allows us to use Hölder's inequality as
\begin{align*}
    \abs{(I)} \le \left(\EE\abs{\phi''(V^{t,\bx}_T)}^{1/H}\right)^H
    \left( \EE\abs{\int_t^{\Tb} \Big\langle\nabla f_s(\Wh_s^{t,\bx}), K^\delta(s,t)-K(s,t) \Big\rangle\ds}^p \,\EE\abs{\int_t^{\Tb} \Big\langle\nabla f_s(\Wh_s^{t,\bx}), K^\delta(s,t) \Big\rangle\ds}^p \right)^{1/p}
\end{align*}
Using Jensen's inequality, estimates \eqref{eq:Payoff_bound} and \eqref{eq:Vol_bound} yield
\begin{align*}
    \abs{(I)} \le C\big(1+\E^{\kappa \norm{\bx}_\infty}\big) \Tb^{2p-2} \int_t^{\Tb} \abs{K^\delta(s,t)-K(s,t)}^p \ds \int_t^{\Tb} \abs{K^{\delta}(s,t)}^p\ds.
\end{align*}
On the one hand, Assumption \ref{assu:kernel} entails
$$
\int_t^{\Tb} \abs{K^{\delta}(s,t)}^p\ds \le \int_t^{\Tb} C(s\vee(t+\delta)-t)^{p(H-\half)}\ds
\le C \int_t^{\Tb} (s-t)^{p(H-\half)}\ds = C\, \frac{1-H}{H} \Tb^\frac{H}{1-H},
$$
where we used $p(H-1/2)+1=\frac{H}{1-H}$.
On the other hand, reasoning as in \eqref{eq:Kdelta-K} we get
\begin{align*}
    \int_t^{\Tb} \abs{K^\delta(s,t)-K(s,t)}^p \ds 
    \le \left(\frac{C\delta}{H-\half}\right)^p \int_t^{t+\delta} (s-t)^{p(H-1/2)} \ds
    = \left(\frac{C\delta}{H-\half}\right)^2  \frac{1-H}{H} \delta^{\frac{H}{1-H}},
\end{align*}
which tends to zero with $\delta$. The same computations show that $(II)$ also goes to zero because $\int_t^T \abs{K(s,t)}^p \ds$ is also bounded.

For the third term, we apply Hölder's inequality followed by Jensen's and Cauchy-Schwarz:
\begin{align*}
    \abs{(III)}
    &\le \left(\EE\abs{ \phi'(V^{t,\bx}_T)}^{1/H}\right)^H
    \left(\EE\abs{\int_t^{\Tb} \Big\langle\nabla^2 f_s(\Wh_s^{t,\bx}), \big(K^\delta(s,t) ,K^\delta(s,t)-K(s,t)\big)\Big\rangle\ds}^{p/2}
    \right)^{2/p}\\
    &\le \left(\EE\abs{ \phi'(V^{t,\bx}_T)}^{1/H}\right)^H \Tb^{p/2-1}  \left(\sup_{s\in[0,\Tb]}\EE\abs{\nabla^2 f_s(\Wh_s^{t,\bx})}^{p/2} \int_t^{\Tb} \abs{K^\delta(s,t)}^{p/2} \abs{K^\delta(s,t)-K(s,t)}^{p/2}\ds\right)^{2/p}.
\end{align*}
The expectations are again bounded thanks to \eqref{eq:Payoff_bound} and \eqref{eq:Vol_bound}. For the integral we will use that $(H-1/2)p/2+1=\frac{1}{2-2H}$ and $(2H-1)p/2+1 = \frac{H}{1-H}$. Once again we leverage on Assumption \ref{assu:kernel} and~\eqref{eq:Diff_Kdelta} to obtain
\begin{align*}
    \int_t^{\Tb} \abs{K^\delta(s,t)}^{p/2} \abs{K^\delta(s,t)-K(s,t)}^{p/2}\ds 
    &\le C\delta^{p/2} \int_t^{t+\delta}  (s-t)^{(H-1/2)p/2} \ds \\
    &=C\delta^{p/2}  \frac{1-H}{H} \delta^\frac{H}{1-H} = \frac{C(1-H)}{H} \delta^{\frac{1+H}{1-H}},
\end{align*}
where the constant $C>0$ may change from line to line. This proves that $(III)$ tends to zero as $\delta\to0$ and, since $(IV)$ converges with the same arguments, this concludes the proof. \qed


\bibliographystyle{siam}
\bibliography{bib}
\end{document}